\documentclass[10pt,reqno]{amsart}
\usepackage{type1cm}        
\usepackage[english]{babel}
\usepackage{dsfont} 
\usepackage{graphicx}
\usepackage{hyperref}
\usepackage{enumerate}
\usepackage{amssymb,empheq,mathrsfs} 
\usepackage{amsthm} 
\usepackage{blkarray}
\usepackage{cleveref}
\usepackage{xcolor}
\usepackage[normalem]{ulem}

\usepackage{moreverb}

\allowdisplaybreaks

\newtheorem{thm}{Theorem}[section]
\newtheorem{lem}[thm]{Lemma}

\newtheorem{cor}[thm]{Corollary}

\newtheorem{open}[thm]{Open Problem}

\theoremstyle{definition}

\newtheorem{dfn}[thm]{Definition}

\newtheorem{exm}[thm]{Example}

\theoremstyle{remark}
\newtheorem{rem}[thm]{Remark}

\newcommand{\exmsymbol}{\hfill$\circ$}

\newcommand{\nset}{\mathds{N}}

\newcommand{\rset}{\mathds{R}}

\newcommand{\diff}{\mathrm{d}}
\newcommand{\divv}{\mathrm{div}\,}
\newcommand{\lin}{\mathrm{lin}\,}

\newcommand{\pos}{\mathrm{Pos}}
\newcommand{\sos}{\mathrm{SOS}}
\newcommand{\war}{\mathrm{War}}

\newcommand{\rank}{\mathrm{rank}\,}

\newcommand{\supp}{\mathrm{supp}\,}

\newcommand{\pol}{\mathrm{Pol}}

\newcommand{\cat}{\mathcal{C}}

\newcommand{\cD}{\mathcal{D}}
\newcommand{\cH}{\mathcal{H}}
\newcommand{\cS}{\mathcal{S}}
\newcommand{\cX}{\mathcal{X}}
\newcommand{\cV}{\mathcal{V}}
\newcommand{\cZ}{\mathcal{Z}}

\newcommand{\sA}{\mathsf{A}}

\newcommand{\fp}{\mathfrak{p}}

\newcommand{\motz}{\mathrm{Motz}}
\newcommand{\rob}{\mathrm{Rob}}
\newcommand{\cl}{\mathrm{CL}}

\begin{document}

\title{Time-dependent moments from partial differential equations and the time-dependent set of atoms}
\author{Ra\'ul E. Curto, Philipp J. di~Dio, Milan Korda and Victor Magron}

\address{Ra\'ul E. Curto, The University of Iowa, Department of Mathematics, 52246, Iowa City, Iowa, U.S. \\ 
Philipp J. di~Dio, University of Konstanz, Department of Mathematics and Statistics, Universit\"atsstra{\ss}e 10, D-78464 Konstanz, Germany \\
Milan Korda, Laboratoire d'analyse et d'architecture des systèmes (LAAS-CNRS), 7 Avenue du Colonel Roche, 31031 Toulouse, France and Faculty of Electrical Engineering, Czech Technical University in Prague, Technick\'a 2, CZ-16626 Prague, Czech Republic \\
Victor Magron, Institut de Mathématiques de Toulouse, 118 Route de Narbonne, 31062 Toulouse France}

\email{raul-curto@uiowa.edu; philipp.didio@uni-konstanz.de; korda@laas.fr; vmagron@laas.fr}

\maketitle

\begin{abstract}
We study the time-dependent moments and associated polynomials arising from the partial differential equation 
$\partial_t f = \nu\Delta f + g\cdot\nabla f + h\cdot f$, and consider in detail the dual equation. \
For the heat equation, we find that several non-negative polynomials which are not sums of squares become sums of squares under the heat equation in finite time. \ We show that every non-negative polynomial in $\rset[x,y,z]_{\leq 4}$ becomes a sum of squares in finite time under the heat equation. \ We solve the problem of moving atoms under the equation $\partial_t f = g\cdot\nabla f + h\cdot f$, with $f_0 = \mu_0$ being a finitely atomic measure. \
The time evolution $\mu_t = \sum_{i=1}^k c_i(t)\cdot \delta_{x_i(t)}$ of the atom positions $x_i(t)$ are described by the transport term $g\cdot\nabla$, and the time-dependent coefficients $c_i(t)$ have an explicit solution depending on $x_i(t)$, $h$, and $\divv g$.
\end{abstract}


\tableofcontents

\section{Introduction}
\label{sec:1}

To study non-negative polynomials the first step is to look at the cone of non-negative polynomials
\[\pos(n,2d) :=  \{f\in\rset[x_1,\dots,x_n]_{\leq 2d} \,|\, f(x)\geq 0\ \text{for all}\ x\in\rset^n\}\]
respectively its dual cone, the closure of the cone of moment functionals.

A second step is to look at (linear) maps (a.k.a. operators)
\[D:\pos(n,2d)\to \pos(n,2d),\]
i.e., the so-called \emph{positivity preservers} \cite{guterman08,borcea11,netzer10}. \ Positivity preservers are fully characterized by their connections to moment sequences and hence we know all positivity preservers if and only if we know all moment sequences.

In the present work we continue the investigation of the heat semi-group $e^{t\Delta}$, $t\geq 0$ \cite{engelNagelSemigroupsGross}. \ In \cite{curtoHeat22} we focused on the moment functionals. \ Now we focus on the non-negative polynomials. \ While positivity preservers are fully characterized, open questions remain. \ For example in \cite{borcea11} the final questions is:
\begin{quote}
Is it true that for any $p\in\pos(n,2d)$ there exists a $f\in\sos(n,2d)$ and a linear partial differential operator $D$ with constant coefficients such that $D \sos(n,2d)\subseteq \pos(n,2d)$ and $Df = p$?
\end{quote}
Our treatment of the heat semi-group is already sufficient to answer the previous question, see \Cref{rem:borceaquestion}.
The approach is general enough to be extended to ellipticity preservers \cite{borcea11} and additional classes of polynomials \cite{branden20,dey21}.
See \Cref{subsec:structure} for the structure of the paper and a summary of our main results and contributions.

\subsection{Static Moments}

Let $\mu$ be a Borel measure on $\rset^n$ for some $n\in\nset$ and $\alpha=(\alpha_1,\dots,\alpha_n)\in\nset_0^n$. \ The $\alpha$-moment $s_\alpha$ of $\mu$ is
\begin{equation}\label{eq:moments}
s_\alpha = \int_{\rset^n} x^\alpha~\diff\mu(x)
\end{equation}
with $x^\alpha := x_1^{\alpha_1}\cdots x_n^{\alpha_n}$.
The \emph{classical moment problem} is: Given finitely or infinitely many real numbers $s_\alpha$ in a sequence $s = (s_\alpha)_{\alpha\in\sA\subseteq\nset_0^n}$, does there exist a Borel measure $\mu$ such that (\ref{eq:moments}) holds for all $\alpha\in\sA$?
If the answer to this question is affirmative, then $\mu$ is called a \emph{representing measure} of $s$ and $s$ is called a \emph{moment sequence}.
If $\sA$ is finite, then $s$ is a \emph{truncated} moment sequence, and if $\sA=\nset_0^n$, then $s$ is called \emph{full} moment sequence.
Additionally, if $\supp\mu\subseteq K$ for some $K\subseteq\rset^n$, then $s$ is called a \emph{$K$-moment sequence}.
A classical tool to investigate moment sequences is the \emph{Riesz functional} $L = L_s:\rset[x_1,\dots,x_n]\to\rset$ defined by $L_s(x^\alpha) := s_\alpha$ and linearly extended to $\rset[x_1,\dots,x_n]$ if $s$ is full or linearly extended to $\mathrm{span}\,\{x^\alpha \,|\, \alpha\in\sA\}$ if $s$ is truncated.
$L:\rset[x_1,\dots,x_n]\to\rset$ is called a \emph{moment functional} if it is represented by some Borel measure $\mu$.
We always assume measures to be non-negative unless specifically denoted as \emph{signed} measures.

Moments are a classical field of research \cite{akhiezClassical,blekSemiOpt,CuFi1,CuFi2,CuFi3,didioCone22,schmudMomentBook,fialkoMomProbSurv,havila35,
havila36,kemper68,kemper71,kemper87,kreinMarkovMomentProblem,landauMomAMSProc,
landau80,lasserreSemiAlgOpt,laurent05,lauren09,lauren09a,shohat43,stielt94,stochel01} and in modern times they are still of interest, e.g.\ because of the following application in optimization.
Let $p\in\rset[x_1,\dots,x_n]$ be a polynomial.
Then 
\begin{equation}\label{eq:min}
\min_{x_0\in K\subseteq\rset^n} p(x_0) = \min_{\substack{\mu:\supp\mu\subseteq K,\\ \mu(K)=1}} \int p(x)~\diff\mu(x) = \min_{\substack{s\ K\text{-moment}\\ \text{sequence},\ s_0=1}} L_s(p),
\end{equation}
since for the first equality we have that for any $x_0\in\rset^n$ the Dirac measure $\delta_{x_0}$ acts as a point evaluation in (\ref{eq:moments}) and the second formulation holds by linearity of the integral using the $s_\alpha$ definition in (\ref{eq:moments}). \ See e.g.\ \cite{lasserreSemiAlgOpt} for more.

A classical result for truncated moment sequences is the \emph{Richter} \cite{richte57} (or \emph{Richter--Rogosinski--Rosenbloom} \cite{richte57,rogosi58,rosenb52}) \emph{Theorem}; see \cite{didioCone22} for a detailed discussion about the historical development.

\begin{thm}[{Richter Theorem 1957 \cite[Satz 4]{richte57}}]\label{thm:richter}
Let $d\in\nset$, $\cV$ be a $d$-dimensional real vector space of measurable real functions $f:\cX\to\rset$ on a measurable space $\cX$, and $L:\cV\to\rset$ be a moment functional, i.e., there exists a measure $\mu$ on $\cX$ such that
\[L(f) = \int_\cX f(x)~\diff\mu(x)\]
for all $f\in\cV$. \ Then there are $c_1,\dots,c_k>0$ and $x_1,\dots,x_k\in\cX$ with $k\leq d$ such that
\[L(f) = \sum_{i=1}^k c_i\cdot f(x_i)\]
for all $f\in\cV$, i.e., we always find a $k$-atomic representing measure $\nu = \sum_{i=1}^k c_i\cdot\delta_{x_i}$ of $L$ with $k\leq d$.
\end{thm}

The points $x_i$ in the $k$-atomic representing measure $\nu$ are called \emph{atoms} and the minimal number $k$ of atoms for fixed $L$ is called the \emph{Carath\'eodory number} $\cat(L)$ (resp.\ $\cat(s)$ for a truncated moment sequence $s$). \ A $k$-atomic representing measure is also called a (\emph{Gauss}) \emph{quadrature rule} \cite{gauss15}. \ Two questions arise naturally in theory and applications:
\begin{enumerate}[(a)]
\item How many atoms $\cat(L)$ are required to represent $L$?

\item Where are the atoms $x_i$ located in a representation of $L$?
\end{enumerate}
Very recent studies about the Carath\'eodory number $\cat(L)$ are \cite{rienerOptima,didio17Cara,didio21HilbertFunction} and the \emph{set of atoms} (or \emph{core variety}) is studied in \cite{fialkow17,didio17w+v+,blekhe20}.

All the studies and references given so far have one thing in common: They study \emph{static} moments, i.e., $s$ is fixed in these studies and properties are only derived from and for $s$. \ Hence, the study of the \emph{moment cone} $\cS_\sA$ (= set of all moment sequences $s$) is only pointwise and collecting or even connecting moment sequences with the same properties in the moment cone is difficult and does not arise naturally.

\subsection{Time-dependent Moments}

Let $n,m\in\nset$. \ We denote by $C_b^\infty(\rset^n,\rset^m)$ the set of \emph{smooth bounded functions}
\[C_b^\infty(\rset^n,\rset^m) := \left\{f\in C^\infty(\rset^n,\rset^m) \,\middle|\, \|\partial^\alpha f\|_\infty < \infty\ \text{for all}\ \alpha\in\nset_0^n\right\}\]
and by $\cS(\rset^n,\rset^m)$ the \emph{Schwartz functions}
\[\cS(\rset^n,\rset^m) := \left\{f\in C^\infty(\rset^n,\rset^m) \,\middle|\, \|x^\alpha \partial^\beta f\|_\infty < \infty\ \text{for all}\ \alpha,\beta\in\nset_0^n\right\}.\]
By $C^d([0,\infty),C_b^\infty(\rset^n,\rset^m))$ we denote all functions $f:\rset^n\times[0,\infty)\to\rset^m$ such that
\begin{enumerate}[(i)]
\item $f(\,\cdot\,,t), \partial_t f(\,\cdot\,,t), \dots, \partial_t^d f(\,\cdot\,,t)\in C_b^\infty(\rset^n,\rset^m)$ for all $t\geq 0$ and
\item $\partial^\alpha f(x,\,\cdot\,)\in C^d([0,\infty),\rset^m)$ for all $x\in\rset^n$ and $\alpha\in\nset_0^n$.
\end{enumerate}
Let $d\in\nset_0$, $\nu = (\nu_1,\dots,\nu_n)^T\in [0,\infty)^n$, $\nu\cdot\Delta = \nu_1\cdot\partial_1^2 + \dots + \nu_n\cdot\partial_n^2$ be the anisotropic Laplace operator, $g = (g_1,\dots,g_n) \in C^d([0,\infty),C_b^\infty(\rset^n,\rset^n))$ be a smooth bounded vector field, $h = (h_{i,j})_{i,j=1}^m\in C^d([0,\infty), C_b^\infty(\rset^n,\rset^{m\times m}))$ be a smooth bounded matrix function, $k = (k_1,\dots,k_m)^T \in\cS(\rset^n,\rset^m)$ be a Schwartz function vector-valued function, and $a\in\rset^n$ be a vector.

Then, by \cite[Thm.\ 2.10]{didio19ENS}, the initial value problem
\begin{equation}\label{eq:pde}
\begin{split}
\partial_t f(x,t) &= \nu\Delta f(x,t) + [ax + g(x,t)]\cdot\nabla f(x,t) + h(x,t)\cdot f(x,t) + k(x,t)\\
f(x,0) &= f_0(x) \in \cS(\rset^n,\rset^m)
\end{split}
\end{equation}
with $(ax + g)\cdot\nabla := (a_1 x_1 + g_1)\cdot\partial_1 + \dots + (a_n x_n + g_n)\cdot\partial_n$ has a unique solution $f\in C^{d+1}([0,\infty),\cS(\rset^n,\rset^m))$. \ Additionally, for $m=1$ and $f_0\geq 0$ we have that $f(\,\cdot\,,t)\geq 0$ for all $t\geq 0$.

For $f_0\in\cS(\rset^n)$ with $f_0\geq 0$ we can calculate the moments
\[s_\alpha(0) := \int_{\rset^n} x^\alpha\cdot f_0(x)~\diff x\]
for all $\alpha = (\alpha_1,\dots,\alpha_n)\in\nset_0^n$. \ Since the time-dependent solution $f$ of (\ref{eq:pde}) is unique, Schwartz function valued, and non-negative for $m=1$ and $f_0\geq 0$, we find that the solution $f$ induces unique time-dependent moments
\begin{equation}\label{eq:timedepMom}
s_\alpha(t) := \int_{\rset^n} x^\alpha\cdot f(x,t)~\diff x
\end{equation}
for all $t\geq 0$.

The time-dependent moments can be defined for nonlinear partial differential equations as well. \ Let us look at a (non-linear) example. \ In \cite[Lem.\ 3.6]{didio19ENS} for Burgers' equation
\[\partial_t f(x,t) = - f(x,t)\cdot\partial_x f(x,t)\]
we calculated the time-dependent moments
\[s_{k,p}(t) := \int_\rset x^k\cdot f(x,t)^p~\diff x\]
of $f(x,t)^p$ for all $k,p\in\nset_0$ and found the explicit expression
\begin{equation}\label{eq:burgerMoments}
s_{k,p}(t) = \sum_{i=0}^k \frac{s_{k-i,p+i}(0)}{i!}\cdot t^i\cdot \prod_{j=0}^{i-1} \frac{(p+j)\cdot (k-j)}{1+(p+j)^2} \quad\in\rset[t]
\end{equation}
which depends only on the initial values $s_{k-i,p+i}(0)$ of $f_0$. \ Hence, despite the fact that for any non-constant $f_0\in\cS(\rset)$ the classical solution of Burgers' equation breaks down in finite time, the moments are \emph{a priori} known for all times $t\in\rset$. \ In \cite[Exm.\ 3.7]{didio19ENS} we calculate, for the one-tooth-function
\[f_0(x) := \begin{cases} 1+x & \text{for}\ x\in [-1,0]\\ 1-x & \text{for}\ x\in [0,1]\\ 0 & \text{else} \end{cases} \quad\quad\geq 0\,,\]
the time-dependent moments $s_{0,1}(t)$, $s_{1,1}(t)$, and $s_{2,1}(t)$, to find
\begin{equation}\label{eq:burgerBreakDown}
\int_\rset (x-t)^2\cdot f(x,t)~\diff x = L_{s(t)}((x-t)^2) = \frac{1}{6} - \frac{2}{15}t^2 \quad\xrightarrow{t\to\pm\infty}\quad -\infty.
\end{equation}
Applying a mollifier $S_\varepsilon$ to $f_0$ gives $S_\varepsilon f_0\in\cS(\rset)$ for all $\varepsilon>0$, and (\ref{eq:burgerBreakDown}) changes continuously with $\varepsilon>0$. \ Thus, (\ref{eq:burgerBreakDown}) holds at least for the time of the existence of the classical solution. \ But the classical solution remains non-negative, which contradicts (\ref{eq:burgerBreakDown}), i.e., Burgers' equation breaks down in finite time. \ In other words, Burgers' equation, as a (non-linear) transport equation, starting with a non-negative $f_0$ implies that the classical solution is also non-negative as long as it exists. \ While from (\ref{eq:burgerMoments}) we see that \emph{a priori} the moments might exist for all times, the derivation of (\ref{eq:burgerMoments}) requires that $f(x,t)$ is a classical solution, i.e., smooth and even Schwartz. \ The break down of the classical solution of Burgers' equation in finite time is therefore not observed through the moments becoming infinite (i.e., they stop to exist), but through the non-negativity property of the representing measure which is encoded in the moments. \ For some $t>0$ we find a non-negative polynomial $p\in\rset[x]$ such that $L_{s(t)}(p)<0$, i.e., (\ref{eq:burgerBreakDown}) shows that non-negativity is not preserved for all times and therefore the classical solution $u$ does not exist for all times.

The Burgers' equation example has two nice features:
\begin{enumerate}[(a)]
\item the time-dependent moments are polynomial in time: $s_{k,p}(t)\in\rset[t]$, and

\item the example to contradict non-negativity of the moments only needs moment up to degree 2, i.e., $0$th, $1$st, and $2$nd moments.
\end{enumerate} 
Studying (\ref{eq:pde}) with general $\nu=(\nu_1,\dots,\nu_n)\in [0,\infty)^n$, $C_b^\infty$-functions $g_i$ and $h_{i,j}$, and Schwartz functions $k_i$, it is evident that neither (a) nor (b) need to be satisfied anymore. \ Therefore, in general we cannot hope for the time-dependent moments $s_\alpha(t)$ to be polynomial in time, depending only on the initial values $s_\alpha(0)$, and to observe certain properties, we can no longer rely on finitely many moments $s_\alpha(t)$. \ We have to explore the computational and theoretical limits of time-dependent moments $s_\alpha(t)$. \ That is, among other things, one purpose of this study. \ We say that a (moment) sequence $s$ \emph{evolves with respect to} or \emph{along the partial differential equation (\ref{eq:pde})} if a representing measure of $s$ evolves with respect to (\ref{eq:pde}). \ Hence, in principle the time-evolution $s(t)$ depends on the choice of representing measure $\mu_0$ of $s(0)$ if $s(0)$ is \emph{indeterminate}, i.e., $s(0)$ has more than one representing measure.

The time-dependent moments $s(t)$ from the heat equation $\partial_t f = \Delta f$ were studied extensively in \cite{curtoHeat22}. \ In \cite{magron20,marx20,korda22,korda22arxiv} moments also have been applied to (non-linear) PDEs. \ In the present work we continue the study in \cite{curtoHeat22} and go beyond the heat equation.

\subsection{The Paper Structure and our Results}\label{subsec:structure}

In the next section (\Cref{sec:dual}) we describe for (\ref{eq:pde}) with $k=0$ the dual action on the polynomials.
We describe the dual action by the dual operator of (\ref{eq:pde}) and show that polynomials in general are no longer polynomials for $t>0$ but at least remain in an algebra $\pol(\rset^n)$ (see (\ref{eq:polyDef}) for the definition).
For $\nu\Delta + g\nabla + h$ with $g(x,t) = g(t)$ and $h(x,t)=h(t)$ we can solve the dual action analytically (see \Cref{cor:dualP}).

We find that $g$ and $h$ have no effect on non-negativity and we therefore study more closely, in \Cref{sec:nonnegpoly}, the (dual) action of the Laplace operator on polynomials. \ We give a simple way to solve the polynomial heat equation in \Cref{thm:polyHeatEqSolution}. \ We collect several results of this action and also present several specific results on non-negative polynomials. \ We show several examples where non-negative polynomials become sums of squares under the heat equation and also give counterexamples, when a non-negative polynomial which is not a sum of squares does not become a sum of squares. \ 
In \Cref{thm:pos34heat} we show for non-negative polynomials $f\in\rset[x,y,z]_{\leq 4}$ that under the heat equation they become sum of squares in finite time. \ 
That is, every non-negative polynomial in $\rset[x,y,z]_{\leq 4}$ is generated by taking a sum of squares in $\rset[x,y,z]_{\leq 4}$ and evolving it along the heat equation with negative times. \ 
All our examples of non-negative polynomials in $\rset[x,y]_{\leq 2d}$ with $d\in\nset_0$ show that they also become sum of squares in finite time. \ 
This observation leads us to \Cref{open1}.

In \Cref{sec:atoms} we investigate the application of (\ref{eq:pde}) with $\nu=0$ and $k=0$ to atomic measures. \ We find that since $g(x,t)\cdot\nabla$ and $h(x,t)$ in general do not commute and therefore a solution cannot be constructed from solutions of the individual operators, for atomic measures we can at first solve the time-evolution with respect to the transport operator $g(x,t)\cdot\nabla$ and then apply the scaling operator $h(x,t)$. \ 
We find that the number of atoms is unchanged. \ We describe the time-evolution of $\mu_t = \sum_{i=1}^k c_i(t)\cdot \delta_{x_i(t)}$ in \Cref{lem:atomMovement} and \Cref{thm:atomMovements}. \ 
For the full moment problem we show that the time-dependent moment sequence (functional) remains on the boundary of the moment cone for all times. \ 
For the truncated moment problem we show that we can enter the interior of the moment cone and that the time-evolution in general depends on the representing measure.

In \Cref{sec:summary} we summarize the results, present final discussions, and state the open problem.

\section{The dual of a partial differential equation acting on polynomials}
\label{sec:dual}

By \cite[Thm.\ 2.10]{didio19ENS} we know that (\ref{eq:pde}) has a unique solution $f\in C^{d+1}([0,\infty),$ $\cS(\rset^n,\rset^m))$ and therefore the time-dependent moments $s_\alpha(t)$ in (\ref{eq:timedepMom}) exist for all times $t\in [0,\infty)$. \ In this section we show that, as in the heat equation \cite{curtoHeat22} the action of the operator $A = \nu\cdot\Delta + g(x,t)\nabla + h(x,t)$ has a dual action $A^*$ acting on the polynomials, i.e., the time-dependency of the solution $f(x,t)$ is moved to a time-dependency of $p(x,t)$:
\[\int p_0(x)\cdot f(x,t)~\diff x = \int p(x,t)\cdot f_0(x)~\diff x\]
for all $t\in [0,\infty)$. \ We will see in \Cref{thm:dualP} that, while $p(x,0)\in\rset[x_1,\dots,x_n]$, we have in general $p(x,t)\not\in\rset[x_1,\dots,x_n]$ for any $t\neq 0$. \ This leads us to introduce the following space of \emph{at most polynomially increasing functions $\pol(\rset^n)$} on $\rset^n$:
\begin{equation}\label{eq:polyDef}\begin{split}
\pol(\rset^n) := \{f\in C^\infty(\rset^n) \,|\,& \text{for all}\ \alpha\in\nset_0^n\ \text{there exists}\ p_\alpha\in\rset[x]\ \text{such}\\
&\text{that}\ |\partial^\alpha f(x)| \leq p_\alpha(x)\ \text{for all}\ x\in\rset^n\}.
\end{split}\end{equation}
We collect some simple properties of $\pol(\rset^n)$.

\begin{lem}
Let $n\in\nset$. \ Then the following hold:
\begin{enumerate}[i)]
\item $\pol(\rset^n)$ is an algebra.

\item $\cS(\rset^n)\subsetneq \pol(\rset^n)$.

\item $\rset[x_1,\dots,x_n]\subsetneq \rset[x_1,\dots,x_n] + C_b^\infty(\rset^n) \subsetneq \pol(\rset^n) \subsetneq \cS(\rset^n)'$.
\end{enumerate}
\end{lem}
\begin{proof}
(i) and (ii) follow immediately from the definition of $\pol(\rset^n)$ in (\ref{eq:polyDef}). \ For (iii) all inclusions $\subseteq$ are clear, so it remains to show that the inclusions are proper $\subsetneq$. \ For the first $\subsetneq$ we have $\sin x_1\in [\rset[x_1,\dots,x_n] + C_b^\infty(\rset^n)]\setminus\rset[x_1,\dots,x_n]$, for the second $\subsetneq$ we have $x_1^2\cdot\sin x_1\in \pol(\rset^n)\setminus [\rset[x_1,\dots,x_n] + C_b^\infty(\rset^n)]$, and the last $\subsetneq$ is clear.
\end{proof}

Let $p\in\pol(\rset^n)$ and $f$ be the unique solution of
\begin{align*}
\partial_t f(x,t) &= \nu\Delta f(x,t) + g(x,t)\cdot\nabla f(x,t) + h(x,t)\cdot f(x,t)\\
f(x,0) &= f_0(x)
\end{align*}
with $f_0\in\cS(\rset^n)$, $g=(g_1,\dots,g_n)^T\in C([0,\infty),C_b^\infty(\rset^n,\rset))$, and $h\in C([0,\infty)$, $C_b^\infty(\rset^n,\rset))$. \ Then
\begin{align*}
\partial_t & \int p(x)\cdot f(x,t)~\diff x\\
&= \int p(x) \cdot [\nu\Delta + g(x,t)\cdot\nabla + h(x,t)] f(x,t)~\diff x
\intertext{is, using integration by parts (since $p\in\pol(\rset^n)$ and $f\in\cS(\rset^n)$),}
&= \int f(x,t)\cdot [\nu\Delta + g(x,t)\cdot\nabla - \divv g(x,t) + h(x,t)]p(x)~\diff x
\end{align*}
The following result shows that the time evolution of $f$ can be shifted to $p$. \ The proof uses techniques of the semi-group approximations of the solution $f$.

\begin{thm}\label{thm:dualP}
Let $d\in\nset_0$, $\nu\geq 0$, $g=(g_1,\dots,g_n)^T\in C^d([0,\infty),C_b^\infty(\rset^n)^n)$, and $h\in C^d([0,\infty), C_b^\infty(\rset^n))$. \ Let $f\in C^{d+1}([0,\infty), \cS(\rset^n))$ be the unique solution of
\begin{equation}\label{eq:pdeSimpleThm}
\begin{split}
\partial_t f(x,t) &= \nu\Delta f(x,t) + g(x,t)\cdot\nabla f(x,t) + h(x,t)\cdot f(x,t)\\
f(x,0) &= f_0(x)
\end{split}
\end{equation}
with $f_0\in\cS(\rset^n)$. \ Let $T>0$. \ Then the unique solution $p_T$ of
\begin{equation}\label{eq:dualPpde}
\begin{split}
\partial_t p_T(x,t) &= \nu\Delta p_T(x,t) - g(x,T-t)\cdot\nabla p_T(x,t)\\
&\qquad + (h(x,T-t)- \divv g(x,T-t))\cdot p_T(x,t)\\
p(x,0) &= p_0(x)\in\pol(\rset^n)
\end{split}
\end{equation}
fulfills $p_T\in C^{d+1}([0,T],\pol(\rset^n))$, and we have
\begin{equation}\label{eq:dualPaction}
\int p_0(x)\cdot f(x,T)~\diff x = \int p_T(x,T)\cdot f_0(x)~\diff x.
\end{equation}
\end{thm}
\begin{proof}
Note, the solution $f$ of (\ref{eq:pdeSimpleThm}) is unique since it is unique on any open bounded set $U\subset\rset^n$, see e.g.\ \cite[Ch.\ 7]{evans10}, and we have $f\in C^{d+1}([0,\infty),\cS(\rset^n))$, see e.g.\ \cite{didio19ENS}. \ Let $N\in\nset$ and $\cZ_N = \{t_0=0 < t_1 < \dots < t_N = T\}$ be a decomposition of $[0,T]$. \ For the operator
\[A(x,t) = \nu\Delta + g(x,t)\cdot\nabla + h(x,t)\],
we then have the dual operator
\[A^*(x,t) = \nu\Delta -g(x,t)\cdot\nabla - \divv g(x,t) + h(x,t).\]
Note, of course, we actually also have to give the domain of the operators $A$ and $A^*$. \ We are working for $A$ on $\cS(\rset^n)$ and for $A^*$ it is therefore sufficient to work on $\pol(\rset^n)$. \ We only have to ensure that $p_T\in\pol(\rset^n)$.

The unique solution $f$ can be approximated in $\cS(\rset^n)$ by the semigroup approach (Trotter \cite{trotter59})
\begin{equation}\label{eq:expApprox}
f(x,T) = \lim_{N\to\infty} \prod_{i=1}^N \exp\left( \int_{t_{i-1}}^{t_i}\!\!\!\!\! A(x,s)~\diff s\right) f_0(x).
\end{equation}
For any operator $B$ we have for the dual $B^*$ the relation $\langle f, Bg\rangle = \langle B^* f, g\rangle$ and for exponentials $\exp(B)$ we have $\langle f, \exp(B) g\rangle = \langle \exp(B^*) f, g\rangle$. \ When we apply these to (\ref{eq:expApprox}) we have to pay attention at the order of the operators, since in general they do not commute and are additionally time-dependent. \ We use the order $\prod_{i=1}^N B_i = B_N B_{N-1} \cdots B_1$ in this formulas. \ Hence, we get
\begin{align*}
\int_{\rset^n} p_0(x)\cdot f(x,T)~\diff x
&= \lim_{N\to\infty} \int_{\rset^n} p_0(x)\cdot \prod_{i=1}^N \exp\left( \int_{t_{i-1}}^{t_i}\!\!\!\!\! A(x,s)~\diff s\right) f_0(x)~\diff x\\
%
%
&= \lim_{N\to\infty} \int_{\rset^n} \left[ \prod_{i=N}^1 \exp\left( \int_{t_{i-1}}^{t_i}\!\!\!\!\! A^*(x,s)~\diff s\right) p_0(x)\right]\cdot  f_0(x)~\diff x\\
&= \int_{\rset^n} p_T(x,T)\cdot f_0(x)~\diff x.
\end{align*}
The last equality holds in the same way as (\ref{eq:expApprox}). \ We set 
\[p_T(x,T) := \lim_{N\to\infty} \prod_{i=N}^1 \exp\left( \int_{t_{i-1}}^{t_i}\!\!\!\!\! A^*(x,s)~\diff s\right) p_0(x)\]
and see that $p_T$ solves (\ref{eq:dualPaction}) since the last equality holds for all $f_0\in\cS(\rset^n)$, i.e., especially for all test functions $f_0\in C_0^\infty(\rset^n)\subset\cS(\rset^n)$. \ This indeed shows that $p_T\in\pol(\rset^n)$ and the time-dependency of $A(x,t)$ combined with the fact that $A(x,t)$ and $A(x,t')$ for $t\neq t'$ in general do not commute shows that we have to take the reverse order in the operator product, i.e., $p_T(x,t)$ solves by substituting $t\mapsto T-t$ in $A^*$ the equation $\partial_t p_T(x,t) = A^*(x,T-t)p_T(x,t)$.
\end{proof}

\begin{rem}
\Cref{thm:dualP} holds with the same proof also for the anisotropic Laplace operator $\nu\cdot\Delta = \nu_1\partial_1^2 + \dots + \nu_n\partial_n^2$ with $\nu = (\nu_1,\dots,\nu_n)\in [0,\infty)^n$. \ Additionally, let $M\in\rset^{n\times n}$ be a symmetric positive-definite matrix. \ Then by a change of coordinates the operator $(\partial_1,\dots,\partial_n) M (\partial_1,\dots,\partial_n)$ can be diagonalized to the anisotropic Laplace operator $\nu\cdot\Delta$. \exmsymbol
\end{rem}

\begin{rem}
In \Cref{thm:dualP} we have seen that the action of (\ref{eq:pde}) on $f_0$ is shifted to $p_0$. \ But since the dual action on $p_0$ is independent on $f_0$, it also provides the action (\ref{eq:pde}) on measures $\mu_0$ instead of functions $f_0$. \ This provides a way to study the set of atoms in \Cref{sec:atoms}.\exmsymbol
\end{rem}

\begin{rem}
In the proof of \Cref{thm:dualP} note that $A = \nu\Delta + g\nabla + h$ is in general an unbounded operator, i.e., it is not defined on all $L^2(\rset^n)$. \ Since $\partial_t f = Af$ with $f_0\in\cS(\rset^n)$ has a unique Schwartz function solution it is therefore sufficient to take the domain of $A$ as $\cD(A) = \cS(\rset^n)$. \ This very restrictive domain enables us in \Cref{thm:dualP} to have $\pol(\rset^n)\subseteq \cD(A^*)$. \ Here, the restriction that $g$ and $h$ are $C_b^\infty$-functions is essential. \ It is not sufficient to have $A$ such that $Af\in\cS(\rset^n)$ for any $f\in\cS(\rset^n)$. \ E.g.\ take $\nu=0$, $g=0$ and $h=x^4$ with $f_0(x)=e^{-x^2}$. \ Then $\partial_t f(x,t) = x^4\cdot f(x,t)$ has as an ordinary differential equation in $t$ with fixed $x$ the unique solution $f(x,t) = e^{-x^2 + t\cdot x^4}$, i.e., $f(\,\cdot\,,t)\notin\cS(\rset^n)$ for any $t>0$. \ Additionally note, that in \cite[Thm.\ 2.2]{denk19} a condition on the generator $A$ is given such that the semi-group $e^{A t}$ maps the Schwartz class $\cS$ to the Schwartz class $\cS$.\exmsymbol
\end{rem}

We have seen in the previous proof that the time-reversal for the dual equation acting on $p_0$ appears because $\Delta$, $g(x,t)\nabla$, and $h(x,t)$ do not commute (pairwise) in general. \ If $g$ and $h$ are time-independent, then of course the time-reversal disappears naturally since no time-dependency exists. \ Another way the time-reversal disappears is, when $\Delta$, $g(x,t)\nabla$, and $h(x,t)$ commute (pairwise), e.g.\ when $g$ and $h$ do not depend on $x$. \ We then have the following explicit solution for $p_t(x,t)$. \ We denote by $\Theta_{t}$ the heat kernel, i.e.,
\[\Theta_{t}(x) = \frac{1}{(4\pi t)^{n/2}}\cdot \exp\left(\frac{-\|x\|^2}{4t}\right)\]
for all $t>0$.

\begin{cor}\label{cor:dualP}
Let $n\in\nset_0$, $\nu\geq 0$, $g=(g_1,\dots,g_n)$, and  $g_1,\dots,g_n,h\in C^d([0,\infty),\rset)$. \ Then for the dual action (\ref{eq:dualPpde}) and (\ref{eq:dualPaction}) in \Cref{thm:dualP} we have
\[p_t(x,t) = e^{H(t)}\cdot [\Theta_{\nu t}*p_0](x+ G(t))\]
with $\Theta_{\nu t}$ the heat kernel, $G(t) := \int_0^t g(s)~\diff s$, and $H(t) := \int_0^t h(s)~\diff s$.
\end{cor}

In \cite{curtoHeat22} the first and the second author studied the time-dependent moments from the heat equation in more detail. \ For the dual action it was especially found that
\begin{align*}
p_0\in\rset[x_1,\dots,x_n]_{\leq d} \qquad &\Rightarrow\qquad [\Theta_{\nu t}*p_0](x)\in \rset[x_1,\dots,x_n]_{\leq d}
\intertext{for all $t\geq 0$ and additionally of course because of the convolution with the non-negative heat kernel}
p_0\geq 0 \qquad &\Rightarrow\qquad [\Theta_{\nu t}*p_0](x)\geq 0
\end{align*}
for all $t\geq 0$. \ Hence, with $p_0\in\rset[x_1,\dots,x_n]$ in \Cref{cor:dualP} we have $p_t(x,t)\in\rset[x_1,\dots,x_n]$ for all $t\in [0,\infty)$ (and even for all $t\in\rset$, see \cite{curtoHeat22}). \ In the general case of \Cref{thm:dualP} we only can ensure $p_t(x,t)\in\pol(\rset^n)$, but not $p_t(x,t)\in\rset[x_1,\dots,x_n]$. \ To see that, let e.g.\ $\nu=0$ and $g=0$, i.e., we have the explicit solution
\[p_t(x,t) = \exp\left(\int_0^t h(x,s)~\diff s\right)\cdot p_0(x)\]
and with $h(x) = \sin(x)$ we have that $p_t(x,t)\not\in\rset[x_1,\dots,x_n]$ for all $t\neq 0$.

Therefore, in case of \Cref{cor:dualP} we have that $g$ and $h$ do not alter the properties of $p_0$ (significantly), but the group action induced by the Laplace operator might give changes. \ We will see that in the next section.

\section{Non-negative polynomials and the heat equation}
\label{sec:nonnegpoly}

Time-dependent moments induced by the heat equation were already studied in \cite{curtoHeat22}. \ There also the dual action on the polynomials was observed. \ We continue this investigation. \ We repeat the essential definitions and results for the convenience of the reader.

\begin{dfn}\label{dfn:fpMono}
Let $d\in\nset_0$. \ We define $\fp_{2d},\fp_{2d+1}\in\rset[x,t]$ by
\begin{align*}
\fp_{2d}(x,t) &:= \sum_{j=0}^d \frac{(2d)!}{(2d-2j)!\cdot j!}\cdot t^j\cdot x^{2d-2j}
\intertext{and}
\fp_{2d+1}(x,t) &:= \sum_{j=0}^d \frac{(2d+1)!}{(2d+1-2j)!\cdot j!}\cdot t^j\cdot x^{2d+1-2j}.
\end{align*}
\end{dfn}

\begin{exm}\label{exm:poly}
We have
\begin{align*}
\fp_0(x,t) &= 1\\
\fp_1(x,t) &= x\\
\fp_2(x,t) &= 2t + x^2\\
\fp_3(x,t) &= 6tx + x^3\\
\fp_4(x,t) &= 12t^2 + 12tx^2 + x^4\\
\fp_5(x,t) &= 60t^2x + 20tx^3 + x^5\\
\fp_6(x,t) &= 120 t^3 + 180 t^2 x^2 + 30t x^4 + x^6\\
&\ \,\vdots \tag*{$\circ$}
\end{align*}
\end{exm}

Straightforward calculations show that $\fp_k$, $k\in\nset_0$, solve the initial value heat equation
\begin{equation}\label{eq:obs}
\begin{split}
\partial_t \fp_k(x,t) &= \partial_x^2 \fp_k(x,t)\\
\fp_k(x,0) &= x^k.
\end{split}
\end{equation}
Hence, by linearity of the heat equation we have the following extension of \Cref{dfn:fpMono} and the observation (\ref{eq:obs}).

\begin{thm}\label{thm:polyHeatEqSolution}
Let $d\in\nset_0$, $n\in\nset$, and $f_0(x) = \sum_{\alpha\in\nset_0^n} c_\alpha\cdot x^\alpha \in\rset[x_1,\dots,x_n]$. \ Then
\begin{equation}\label{eq:fpDefPoly}
\fp_{f_0}(x,t) := \sum_{\alpha\in\nset_0^n} c_\alpha\cdot \fp_{\alpha_1}(x_1,t)\cdots \fp_{\alpha_n}(x_n,t) \quad\in\rset[x_1,\dots,x_n,t]
\end{equation}
solves the initial value heat equation
\begin{equation}\label{eq:heat}\begin{split}
\partial_t f(x,t) &= \Delta f(x,t)\\
f(x,0) &= f_0(x).
\end{split}
\end{equation}
\end{thm}
\begin{proof}
By linearity of the Laplace operator $\Delta$ it is sufficient to look at $f_0(x) = x^\alpha$ for $\alpha\in\nset_0^n$. \ By (\ref{eq:obs}) we already have $\partial_t \fp_{\alpha_i}(x_i) = \partial_i^2 \fp_{\alpha_i}(x)$ and hence
\begin{align*}
\partial_t \fp_{f_0}(x,t) &= \partial_t[\fp_{\alpha_1}(x_1,t)\cdots\fp_{\alpha_n}(x_n,t)]\\
&= [\partial_t \fp_{\alpha_1}(x_1,t)]\cdot \fp_{\alpha_2}(x_2,t)\cdots\fp_{\alpha_n}(x_n,t)\\ &\qquad + \dots + \fp_{\alpha_1}(x_1,t)\cdots\fp_{\alpha_{n-1}}(x_{n-1},t)\cdot [\partial_t \fp_{\alpha_n}(x_n,t)]\\
&= [\partial_1^2 \fp_{\alpha_1}(x_1,t)]\cdot \fp_{\alpha_2}(x_2,t)\cdots\fp_{\alpha_n}(x_n,t)\\ &\qquad + \dots + \fp_{\alpha_1}(x_1,t)\cdots\fp_{\alpha_{n-1}}(x_{n-1},t)\cdot [\partial_n^2 \fp_{\alpha_n}(x_n,t)]\\
&= \Delta \fp_{f_0}(x,t).\qed
\end{align*}
\end{proof}

Note, in \Cref{dfn:fpMono} we defined $\fp$ for the monomials $x^d$, $d\in\nset_0$, and in (\ref{eq:fpDefPoly}) we define $\fp$ by linearity for any $f_0\in\rset[x]$. \ Hence, if $f_0(x) = x^d$ then both definitions coincide: $\fp_{f_0} = \fp_d$. \ The same shall hold for the multivariate case to keep the notation simple.

The unique solution of (\ref{eq:heat}) can be written as the convolution with the heat kernel $\Theta_t$, i.e.,
\begin{equation}\label{eq:heatConvolution}
\fp_{f_0}(x,t) = (\Theta_t * f_0)(x)
\end{equation}
for all $t>0$, or, since $\Delta^{\lfloor\frac{1}{2}\deg f_0\rfloor + 1} f_0 = 0$, we can also write the solution as
\begin{equation}\label{eq:taylorSemigroup}
\fp_{f_0}(\,\cdot\,,t) = e^{t\Delta} f_0 = \sum_{k=0}^\infty \frac{t^k}{k!}\Delta^k f_0 = \sum_{k=0}^{\lfloor\frac{1}{2}\deg f_0\rfloor} \frac{t^k}{k!}\Delta^k f_0.
\end{equation}
For more on one-parameter semigroups see \cite{engelNagelSemigroupsGross}.
Note that the polynomial heat equation (\ref{eq:heat}) has a unique polynomial solution (\ref{eq:taylorSemigroup}) for all $t\in\rset$, contrary to the $L^2$-heat equation; i.e., $p_0\in L^2$, which can in general only be solved for $t\geq 0$.
In connection with (\ref{eq:taylorSemigroup}) we also want to mention the works \cite{guterman08,netzer10}.

\begin{exm}[Motzkin polynomial \cite{motzkin65}]\label{exm:motzkin}
Let
\[f_{\text{Motz}}(x,y) = 1 - 3x^2y^2 + x^4 y^2 + x^2 y^4 \in\rset[x,y]\]
be the Motzkin polynomial. \ Then by \Cref{dfn:fpMono} (resp.\ \Cref{exm:poly}) we have the substitutions
\begin{align*}
x^2 &\mapsto 2t + x^2, \qquad x^4 \mapsto 12t^2 + 12tx^2 + x^4,\\
y^2 &\mapsto 2t + y^2, \qquad y^4 \mapsto 12t^2 + 12ty^2 + y^4
\end{align*}
and get
\begin{align*}
\fp_{\text{Motz}}(x,y,t) &= 1 - 3(2t + x^2)(2t + y^2) + (12t^2 + 12tx^2 + x^4)(2t+y^2)\\
&\quad + (2t+x^2)(12t^2 + 12ty^2 + y^4)\\
&= 1 - 12t^2 + 48t^3 + 6t(-1+6t)(x^2+y^2) 
 + (-3+24t) x^2 y^2\\
&\quad + 2t(x^4 + y^4) + x^4 y^2 + x^2 y^4
\end{align*}
for all $t\in\rset$.\exmsymbol
\end{exm}

Since the heat kernel is a Schwartz function the convolution with polynomials is well-defined and uniqueness of the solution of the heat equation shows that
\begin{equation}\label{eq:convolution}
(\Theta_t * f_0)(x) = \fp_{f_0}(x,t)
\end{equation}
holds for all $f_0\in\rset[x_1,\dots,x_n]$. \ That the heat kernel preserves polynomials holds for all convolution kernels which are integrable with respect to polynomials. \ To make the paper self-contained, let us briefly state and prove this known fact.

\begin{thm}\label{thm:polyConv}
Let $d\in\nset_0$ and $\rho$ be a kernel such that $\int_{\rset^n} y^\alpha\cdot \rho(y)~\diff y$ is finite for all $\alpha\in\nset_0^n$ with $|\alpha|\leq d$, then
\[\cdot\,*\rho:\rset[x_1,\dots,x_n]_{\leq d} \to \rset[x_1,\dots,x_n]_{\leq d}.\]
\end{thm}
\begin{proof}
Let $p\in\rset[x_1,\dots,x_n]_{\leq d}$. \ Then from
\begin{align*}
(p*\rho)(x) = \int_{\rset^n} p(x-y)\cdot\rho(y)~\diff y
\end{align*}
and expanding $p(x-y)$ in the right side gives the assertion including the degree bound $\deg (p*\rho) \leq d$.
\end{proof}

Since the heat kernel is non-negative, the convolution of the heat kernel with a non-negative polynomial gives again a non-negative polynomial. \ Denote by $\pos(n,d)$ the set of all non-negative polynomials on $\rset^n$ with degree at most $d\in\nset_0$:
\[\pos(n,d) := \{ p\in \rset[x_1,\dots,x_n]_{\leq d} \,|\, p(x)\geq 0\ \text{for all}\ x\in\rset^n\}.\]
Non-negativity and \Cref{thm:polyConv} gives the following.

\begin{cor}\label{cor:nonneg}
Let $n\in\nset$, $d\in\nset_0$, and $f_0\in\pos(n,d)$. \ Then
\[\fp_{f_0}(\,\cdot\,,t)\in\pos(n,d)\]
for all $t\geq 0$. \ Especially, if $f_0\neq 0$ then $\fp_{f_0}(\,\cdot\,,t) > 0$ on $\rset^n$ for all $t>0$.
\end{cor}

\begin{cor}
Let $n\in\nset$ and $f_0\in\rset[x_1,\dots,x_n]$. \ Assume there exist $t>0$ and a point $\xi\in\rset^n$ such that $\fp_{f_0}(\xi,t)<0$. \ Then $f_0\not\in\pos(n,d)$.
\end{cor}

Hence, recalling the Motzkin polynomial in \Cref{exm:motzkin} we find that $\fp_{\text{Motz}}(\,\cdot\,,t)\in\pos(2,6)$ for all $t\geq 0$. \ In the usual case (e.g.\ $f_0\in L^2(\rset^n)$) the convolution with the heat kernel has a smoothing effect on $f_0$, i.e., the regularity increases and $\Theta_t*f_0$ is a $C^\infty$-function. \ But in the polynomial case we already started with a $C^\infty$-function and this kind of regularity does not change. \ We have even seen that for $f_0\in\rset[x_1,\dots,x_n]$ the function $\fp_{f_0}$ remains a polynomial for all $t\in\rset$. \ But from the Motzkin polynomial in \Cref{exm:motzkin} we observe something additional. \ The Motzkin polynomial was of course the first non-negative polynomial found that is not a sum of squares. \ However, the heat kernel has another ``smoothing effect'' for this polynomial, as seen in the following continuation. \ We denote by $\sos(n,d)$ the set of all sums of squares in $n\in\nset$ variables of degree less or equal to $d\in\nset_0$:
\[\sos(n,d) := \{p\in\rset[x_1,\dots,x_n]_{\leq d} \,|\, p\ \text{is a sum of squares}\}.\]

\begin{exm}[Motzkin polynomial, \Cref{exm:motzkin} continued]\label{exm:motzkinCont}
We have
\begin{align*}
\fp_{\text{Motz}}(x,y,1) &= 37\cdot \left(1 - \frac{11}{148}x^2 - \frac{11}{148} y^2 \right)^2 + \frac{71}{2}\cdot \left(x - \frac{4}{71}x y^2 \right)^2\\
&\quad + \frac{71}{2}\cdot \left(y - \frac{4}{71}x^2 y \right)^2 + \frac{57}{2} x^2 y^2 + \frac{1063}{592}\cdot \left(x^2 + \frac{27}{1063} y^2 \right)^2\\
&\quad + \frac{3815}{2126}\cdot y^4 + \frac{63}{71}\cdot x^4 y^2 + \frac{63}{71} x^2 y^4 \quad\in\sos(2,6)
\end{align*}
i.e., $\fp_{\text{Motz}}(\,\cdot\,,1)$ is by \Cref{cor:nonneg} not just non-negative, but in fact a sum of squares. \ This relation can easily be obtained e.g.\ by the use of Macaulay2 \cite{mac2} and the SumsOfSquares package \cite{cifuen20}. \ In fact, additional calculations indicate that
\[\fp_{\text{Motz}}(\,\cdot\,,t) \in \begin{cases} \pos(2,6)\setminus \sos(2,6) & \text{for}\ t\in [0,T_\motz),\ \text{and}\\ \sos(2,6) & \text{for}\ t\in [T_\motz,\infty), \end{cases}\]
with
\[ \frac{31\,998}{1\,000\,000} \quad<\quad T_\motz \quad<\quad \frac{31\,999}{1\,000\,000}.\]
The choice of the intervals $[0,T_\motz)$ and $[T_\motz,\infty)$ is clear since $\sos(2,6)$ is closed and $\fp_{f_0}(\,\cdot\,,t)$ continuous in $t$, i.e., $\fp_{\text{Motz}}(\,\cdot\,,T_\motz)\in\sos(2,6)$.\exmsymbol
\end{exm}

$\fp_{\text{Motz}}:[0,\infty)\ni t\mapsto\fp_{\text{Motz}}(\,\cdot\,,t)\in\pos(2,6)$ is a continuous path through the cone of non-negative polynomials. \ The following result shows that once $\fp_{f_0}$ enters $\sos(n,d)$ e.g.\ at time $t_0\geq 0$, then it stays in $\sos(n,d)$ for all $t\geq t_0$.

\begin{thm}\label{thm:sos}
Let $\rho\geq 0$ be a kernel such that $\int_{\rset^n} y^\alpha\cdot \rho(y)~\diff y$ is finite for all $\alpha\in\nset_0^n$ with $|\alpha|\leq d$, then
\[\cdot\,*\rho: \sos(n,d)\to \sos(n,d).\]
\end{thm}
\begin{proof}
Let $p\in \sos(n,d)$, i.e., there exists a symmetric $Q\in\rset^{N\times N}$ with $N = \binom{n+d}{d}$ such that $p(x) = (x^\alpha)_{\alpha}^T \cdot Q \cdot (x^\alpha)_{\alpha}$ where $(x^\alpha)_{\alpha}$ is the vector of all monomials $x^\alpha$ with $|\alpha|\leq d$. \ We then have
\begin{align*}
(p*\rho)(x) &= \int_{\rset^n} p(x-y)\cdot \rho(y)~\diff y\\
&= \int_{\rset^n} ((x-y)^\alpha)_{\alpha}^T \cdot Q \cdot ((x-y)^\alpha)_{\alpha}\cdot\rho(y)~\diff y
\intertext{and by Richter's Theorem \cite{richte57} we can replace $\rho(y)~\diff y$ by a finitely atomic representing measure $\mu = \sum_{i=1}^k c_i\cdot\delta_{y_i}$ with $c_i>0$ and get}
&= \sum_{i=1}^k c_i\cdot ((x-y_i)^\alpha)_{\alpha}^T \cdot Q \cdot ((x-y_i)^\alpha)_{\alpha} \in \sos(n,d).\qedhere
\end{align*}
\end{proof}

Besides sums of squares, other non-negative polynomials are linear combinations of even powers of linear forms, i.e., they have a \emph{Waring decomposition}
\[p(x) = \sum_{i=1}^k (a_i\cdot x)^{d}\]
with even $d\in\nset$, $a_i\in\rset^{n+1}$, and $a_i\cdot x := a_{i,0} + a_{i,1} x_1 + \dots + a_{i,n} x_n$ \cite{reznick92}. \ We denote by $\war(n,d)$ the (Waring) cone of all these polynomials, i.e., we have the proper inclusions
\[\war(n,d) \quad\subsetneq\quad \sos(n,d) \quad\subsetneq\quad \pos(n,d).\]
In \Cref{thm:polyConv} we have seen that convolution preserves being a polynomial, in \Cref{thm:sos} we have seen that convolution preserves being a sum of squares, and the next result shows that convolution also preserves being in the Waring cone.

\begin{thm}\label{thm:waring}
Let $\rho\geq 0$ be a kernel such that $\int_{\rset^n} y^\alpha\cdot\rho(y)~\diff y$ is finite for all $\alpha\in\nset_0^n$, then $\cdot\,*\rho:\war(n,d)\to \war(n,d)$.
\end{thm}
\begin{proof}
Let $p\in \war(n,d)$, i.e., $p(x) = \sum_{i=1}^k (a_i\cdot x)^d$. \ Then
\begin{align*}
(p*\rho)(x) &= \int_{\rset^n} p(x-y)\cdot \rho(y)~\diff y\\
&= \int_{\rset^n} \sum_{i=1}^k (a_i\cdot (x-y))^d\cdot \rho(y)~\diff y
\intertext{and by Richter's Theorem \cite{richte57} we can replace $\rho(y)~\diff y$ by a finitely atomic representing measure $\mu = \sum_{j=1}^l c_j\cdot\delta_{y_j}$ with $c_j>0$ and get}
&= \sum_{j=1}^l \sum_{i=1}^k c_j\cdot (a_i\cdot (x-y_j))^d \in \war(n,d).\qedhere
\end{align*}
\end{proof}

The Motzkin polynomial was the first polynomial found to be a non-negative polynomial which is not a sum of squares. \ But others have been identified \cite{marshallPosPoly}. \ We want to investigate some of these in chronological order.

\begin{exm}[Robinson polynomial \cite{robinson69}]\label{exm:robinson}
Let
\[f_{\rob}(x,y) = 1 - x^2 - y^2 - x^4 + 3x^2 y^2 - y^4 + x^6 - x^4 y^2 - x^2 y^4 + y^6	\]
be the Robinson polynomial, i.e., $f_{\rob}\in\pos(2,6)\setminus\sos(2,6)$. \ Then by a direct calculation using Macaulay2 with the SumsOfSquares package similar to the Motzkin polynomial we find $\fp_\rob(\,\cdot\,,1)\in\sos(2,6)$ and by \Cref{thm:sos} we have
\[\fp_\rob(\,\cdot\,,t) \in \begin{cases} \pos(2,6)\setminus \sos(2,6) & \text{for}\ t\in [0,T_\rob),\ \text{and}\\ \sos(2,6) & \text{for}\ t\in [T_\rob,\infty), \end{cases}\]
with
\[ \frac{20\,946}{1\,000\,000} \quad<\quad T_\rob \quad<\quad \frac{20\,947}{1\,000\,000}. \tag*{$\circ$}\]
\end{exm}

\begin{thm}\label{exm:choilam}
Let
\[f_{\cl}(x,y,z) := 1 - 4xyz + x^2 y^2 + x^2 z^2 + y^2 z^2 \quad\in\pos(3,4)\setminus\sos(3,4)\]
be the Choi--Lam polynomial \cite{choi77}. \ Then
\[\fp_{\cl}(\,\cdot\,,t) \in \begin{cases} \pos(3,4)\setminus \sos(3,4) & \text{for}\ t\in [0,1/9),\ \text{and}\\ \sos(3,4) & \text{for}\ t\in [1/9,\infty), \end{cases}\]
\end{thm}
\begin{proof}
We have
\begin{align*}
\fp_\cl(x,y,z,t) &= 1 - 4xyz + (2t + x^2)(2t + y^2) + (2t + x^2)(2t + z^2)\\
&\quad + (2t + y^2)(2t + z^2)\\
&= 1 + 12t^2 - 4xyz + 4t(x^2 + y^2 + z^2) + x^2 y^2 + x^2 z^2 + y^2 z^2\\
&= f_\cl(x,y,z) + 12t^2 + 4t(x^2 + y^2 + z^2)\\
&= v(x,y,z)^\top G_t v(x,y,z),
\end{align*} 
with the monomial vector $v(x,y,z):=(1,x,y,z,xy,xz,yz,x^2,y^2,z^2)^\top$ and the Gram matrix $G_t =$
\[\scriptsize\begin{blockarray}{ccccccccccc}
  & 1 & x & y & z & xy & xz & yz & x^2 & y^2 & z^2 \\
  \begin{block}{c(cccccccccc)}
    1  & 1+12t^2 & & & & & & & \delta_x & \delta_y & \delta_z \\
    x  &  & 4t-2\delta_x & & & & & a_x \\
    y  &  &  & 4t-2\delta_y & & & a_y \\
    z  &  &  &  & 4t-2\delta_z & a_z\\
    xy &  &  &  & a_z & 1-2\varepsilon_z\\
    xz &  &  & a_y &  &  & 1-2\varepsilon_y\\
    yz &  & a_x & & & & & 1-2\varepsilon_x\\
    x^2& \delta_x & & & & & & & 0 & \varepsilon_z & \varepsilon_y\\
    y^2& \delta_y & & & & & & & \varepsilon_z & 0 & \varepsilon_x\\
    z^2& \delta_z & & & & & & & \varepsilon_y & \varepsilon_x & 0\\
  \end{block}
\end{blockarray}\]
with $a_x + a_y + a_z = -2$ and $t\in\rset$.
The Gram matrix is of course not unique, but for $-4xyz$ there are only $a_x$, $a_y$, and $a_z$.
And since $f_\cl$ does not contain $x^4$, $y^4$, or $z^4$ we also have zeros at the positions $(x^2,x^2)$, $(y^2,y^2)$, and $(z^2,z^2)$ (the diagonal of the $\varepsilon$-block of the columns and rows $x^2$, $y^2$, and $z^2$).

At $t=1/9$ we have
\[\fp_\cl(x,y,z,1/9) = \frac{31}{27} + \left(xy - \frac{2}{3}z \right)^2 + \left(xz - \frac{2}{3}y \right)^2 + \left(yz - \frac{2}{3}x \right)^2 \in\sos(3,4).\]
It remains to show that $\fp_\cl$ for $t<1/9$ it is not a sum of squares, i.e., no Gram matrix representation fulfills $G_t\succeq 0$.

The coefficient of $x^2$ can be written in the $(1,x^2)$ or $(x,x)$ entries, i.e., we have the free parameter $\delta_x$.
Similarly with $y^2$ and $z^2$.
Hence, in every Gram matrix representation of $\fp_\cl$, the submatrices
\begin{equation}\label{eq:submatrix}
\begin{pmatrix} 4t-2\delta_x & a_x\\ a_x & 1-2\varepsilon_x\end{pmatrix}, \begin{pmatrix} 4t-2\delta_y & a_y\\ a_y & 1-2\varepsilon_y\end{pmatrix}, \begin{pmatrix} 4t-2\delta_z & a_z\\ a_z & 1-2\varepsilon_z\end{pmatrix}
\end{equation}
appear.
We show $\not\succeq 0$ for at least one of them if $t<1/9$. \ So assume $t<1/9$ and assume to the contrary that $G_t\succeq 0$.
Then $\varepsilon_x = \varepsilon_y = \varepsilon_z = 0$ and $\delta_x = \delta_y = \delta_z = 0$.
Since $a_x + a_y + a_z = -2$ we have that at least one of $a_x$, $a_y$, or $a_z$ is $\leq -\frac{2}{3}$.
Without loss of generality let $a_x\leq -\frac{2}{3}$. \ Then
\[\det \begin{pmatrix} 4t & a_x\\ a_x & 1\end{pmatrix} = 4t - a_x^2 \leq 4t - \frac{4}{9} < 0\]
and hence $G_t\not\succeq 0$ for any Gram matrix representation of $\fp_\cl$ with $t<1/9$.
\end{proof}

Macaulay2 calculations with the SumsOfSquares package suggest
\[\frac{1}{9} - 7\cdot 10^{-9}\quad<\quad \min \{t\geq 0 \,|\, \fp_\cl(\,\cdot\,,t)\in\sos(3,4)\} \quad<\quad \frac{1}{9} - 6\cdot 10^{-9}.\]
That is close to the exact value of $1/9$ found in \Cref{exm:choilam}.

\begin{exm}[Schm\"udgen polynomial \cite{schmud79}]\label{exm:schmudgen}
The polynomial
\begin{align*}
f_{\text{Schm}}(x,y) &= (y^2-x^2)x(x+2)[x(x-2)+2(y^2-4)]\\
&\quad + 200 [(x^3-4x)^2 + (y^3-4y)^2] \quad\in\pos(2,6)\setminus\sos(2,6)
\end{align*}
is the Schm\"udgen polynomial and we find $\fp_{\text{Schm}}(\,\cdot\,,1)\in\sos(2,6)$. \ In fact, Macaulay2 calculations with the SumsOfSquares package and \Cref{thm:sos} shows that $\fp_{\text{Schm}}(\,\cdot\,,t)\in\sos(2,6)$ for all $t\geq 2\cdot 10^{-4}$.\exmsymbol
\end{exm}

\begin{exm}[Berg--Christensen--Jensen polynomial \cite{berg79}]\label{exm:berg}
The Berg--Christensen--Jensen polynomial
\[f_{\text{BCJ}}(x,y) = 1 - x^2 y^2 + x^4 y^2 + x^2 y^4 \in\pos(2,6)\setminus\sos(2,6)\]
is connected to the Motzkin polynomial $f_{\motz}$ (\Cref{exm:motzkin}) by
\[f_{\text{BCJ}}(x,y) = f_\motz(x,y) + 2x^2y^2\]
and hence from \Cref{thm:sos} we see that
\[\fp_{\text{BCJ}}(\,\cdot\,,t)\in\sos(2,6)\]
for all $t\geq \frac{1}{6}$.\exmsymbol
\end{exm}

After all these historic examples let us have a look at a more modern example from the vast literature of non-negative polynomials which are not sums of squares.

\begin{exm}[Harris polynomial {\cite[$R_{2,0}$ in Lem.\ 5.1 and 6.8]{harris99}}]\label{exm:harris}
Let
\begin{align*}
f_{\text{Har}}(x,y) &= \phantom{+}\, 16x^{10} - 36 x^8 y^2 + 20 x^6 y^4 + 20 x^4 y^6 - 36 x^2 y^8 + 16 y^{10}\\
&\quad - 36 x^8\phantom{^0} + 57 x^6 y^2 - 38 x^4 y^4 + 57 x^2 y^6 - 36 y^8\\
&\quad + 20 x^6\phantom{^0} - 38 x^4 y^2 - 38 x^2 y^4 + 20 y^6\\
&\quad + 20 x^4\phantom{^0} + 57 x^2 y^2 + 20 y^4\\
&\quad - 36 x^2\phantom{^0} - 36 y^2\\
&\quad + 16
\end{align*}
be the Harris polynomial, i.e., $f_{\text{Har}} = R_{2,0}\in\pos(2,10)\setminus\sos(2,10)$. \ With \Cref{exm:poly},
\begin{align*}
\fp_8(x,t) &= 1680 t^4 + 3360 t^3 x^2 + 840t^2 x^4 + 56 t x^6 + x^8,
\intertext{and}
\fp_{10}(x,t) &= 30240 t^5 + 75600t^4 x^2 + 25200t^3 x^4 + 2520t^2 x^6 + 90tx^8 + x^{10}
\end{align*}
we calculate $\fp_{\text{Har}}$ and find $\fp_{\text{Har}}(\,\cdot\,,1)\in\sos(2,10)$. \ In fact, Macaulay2 calculations and \Cref{thm:sos} show that $\fp_{\text{Har}}(\,\cdot\,,t)\in\sos(2,10)$ for all $t \geq 8\cdot 10^{-4}$.\exmsymbol
\end{exm}

All examples so far become a sum of squares for large $t$. \ However, this is not in general true.

\begin{lem}\label{lem:notSOS}
Let $n,d\in\nset$ and
\[f(x) = \sum_{|\alpha|\leq 2d} a_\alpha\cdot x^\alpha \in\pos(n,2d)\]
such that
\[f_{2d}(x) := \sum_{|\alpha|=2d} a_\alpha\cdot x^\alpha \not\in\sos(n,2d),\]
then $\fp_f(\,\cdot\,,t)\in\pos(n,2d)\setminus\sos(n,2d)$ for all $t\geq 0$.
\end{lem}
\begin{proof}
Assume there is a $t\geq 0$ such that
\[\fp_f(x,t)=\sum_{i=1}^k\left(\sum_{|\alpha|\leq d} c_{i,\alpha}(t)\cdot x^\alpha\right)^2 = \sum_{|\alpha|\leq 2d} a_\alpha(t)\cdot x^\alpha \in\sos(n,2d).\]
Since by \Cref{dfn:fpMono} we have $a_\alpha(t) = a_\alpha(0)$ for all $\alpha\in\nset_0^n$ with $|\alpha|=2d$ the sum of squares decomposition of $\fp_f(\,\cdot\,,t)$ gives
\[f_{2d}(x)=\sum_{i=1}^k\left(\sum_{|\alpha|=d} c_{i,\alpha}(t)\cdot x^\alpha\right)^2\in\sos(n,2d)\]
which contradicts the assumption $f_{2d}\not\in\sos(n,2d)$.
\end{proof}

The following example shows that the condition $f_{2d}\in \sos(n,2d)$ in \Cref{lem:notSOS} is necessary, but not sufficient.
Just take the polynomial $f(w,x,y,z) = z^6 - 3x^2 y^2 z^2 + x^4 y^2 + x^2 y^4 + w^8\in\pos(3,8)\setminus\sos(3,8)$.

\begin{exm}\label{exm:nonSOS}
Let $f(x,y,z) = z^6 - 3x^2 y^2 z^2 + x^4 y^2 + x^2 y^4\in\pos(3,6)\setminus\sos(3,6)$ be the homogeneous Motzkin polynomial. \ Then $\fp_f(\,\cdot\,,t)\in\pos(3,6)\setminus\sos(3,6)$ for all $t\geq 0$.\exmsymbol
\end{exm}

\begin{rem}
The result in \Cref{lem:notSOS} also holds for $f_{2d}\not\in\war(n,2d)$, i.e., $\fp_f(\,\cdot\,,t)\in\pos(n,2d)\setminus\war(n,2d)$.\exmsymbol
\end{rem}

\begin{rem}\label{rem:borceaquestion}
In the Introduction we mentioned the final question in \cite{borcea11}.
For any given $p\in\pos(n,2d)$ is there a
\begin{equation}\label{eq:Ddfn}
D = \sum_{\alpha\in\nset_0^n} c_\alpha \partial^\alpha \qquad\text{with}\qquad c_\alpha\in\rset
\end{equation}
and a $f\in\sos(n,2d)$ such that
\[Df = p \qquad\text{and}\qquad D\sos(n,2d)\subseteq\pos(n,2d)?\]
From \Cref{lem:notSOS} it is clear that any $D$ as in (\ref{eq:Ddfn}) does not affect the highest degree part of $p$, as differentiation changes only to lower degrees since $c_0\in\rset$.
That is, choosing $p = z^6 - 3x^2 y^2 z^2 + x^4 y^2 + x^2 y^4$ as in \Cref{exm:nonSOS} gives a simple example such that neither
\[Dp \in \sos(n,2d) \qquad\text{nor}\qquad p\in \tilde{D}\sos(n,2d)\]
for any $D$ or $\tilde{D}$ as in (\ref{eq:Ddfn}).
The reader should especially note that $D\sos(n,2d)\subseteq\pos(n,2d)$ in the question is not necessary to give a general counterexample.
Here, the reader will also sees why it was sufficient from the beginning to look only at $D = e^{t\Delta}$.
\exmsymbol
\end{rem}

While we have seen in \Cref{lem:notSOS} and \Cref{exm:nonSOS} that there are non-negative polynomials which do not become sum of squares under the heat equation, the following result  shows that any non-negative polynomial becomes asymptotically close to $\sos(\rset^n)$ under the heat equation, that is, the constant polynomial becomes an attractor of the polynomial heat equation.

\begin{lem}\label{lem:sosAsym}
Let $n\in\nset$ and $f\in\pos(\rset^n)$ with $\deg f = 2d$ for some $d\in\nset$. \ Then
\[\lim_{t\to\infty} \fp_f(x,t)\cdot t^{-d} = c > 0.\]
\end{lem}
\begin{proof}
Let $k\in\nset$. \ It is easy to see that $\Delta^k g$ is constant on $\rset^n$ for all $g\in\rset[x_1,\dots,x_n]_{\leq 2k}$ and even equal to zero for all $g\in\rset[x_1,\dots,x_n]_{\leq 2k-1}$.

Since $f\in\pos(\rset^n)$ with $\deg f = 2d$ we have that the homogeneous part $f_{2d}$ of $f$ of degree $2d$ is non-zero and non-negative on $\rset^n$. \ Let $S$ be the unit sphere in $\rset^n$. \ Since $f_{2d}\in\pos(\rset^n)\setminus\{0\}$ we have
\[\int_S f_{2d}(x)~\diff x > 0\]
and by \cite[Cor.\ 1]{iltyak98} we have $\Delta^m f_{2d} \neq 0$ for all $m=1,\dots,d$. \ Hence, $\partial_t^d \fp_f(x,t) = \Delta^d \fp_f(x,t) = \Delta^d f_{2d}(x) = c > 0$ which proves the statement.
\end{proof}

Note that \cite[Cor.\ 1]{iltyak98} can also be replaced by \cite[Thm.\ 19.16]{schmudMomentBook}.
There it is shown that $\Delta^d:\rset[x_1,\dots,x_n]_{\leq 2d}\to\rset$ is a strictly positive moment functional.

\Cref{lem:notSOS} and \Cref{exm:nonSOS} have only been proven here for $n\geq 3$, but for $n=2$ no counterexample has been found.
In fact, every time evolution of the polynomials $f\in\pos(\rset^2)\setminus\sos(\rset^2)$ in the Examples \ref{exm:motzkinCont}, \ref{exm:robinson}, \ref{exm:schmudgen}, and \ref{exm:berg} enters $\sos(\rset^2)$ and by \Cref{thm:sos} never leaves $\sos(\rset^2)$ again.
It is open if for $n=2$ every time evolution of a non-negative polynomial becomes a sum of squares; see \Cref{open1}.
For $\pos(3,4)$ this question is answered affirmative by the following theorem.

\begin{thm}\label{thm:pos34heat}
Let $f\in\pos(3,4)$. \ Then there exists a $\tau_f \in [0,\infty)$ such that
\[\fp_f(\,\cdot\,,t)\in\sos(3,4)\]
for all $t\geq \tau_f$.
\end{thm}
\begin{proof}
Let $f_4$ be the leading term (homogeneous part of highest degree $4$) of $f$.
Since $f_4$ is a homogeneous polynomial of degree $4$ in three variables it is a sum of squares. \ By linearity let
\[f_4(x,y,z) = (ax^2 + bxy + cxz + dy^2 + eyz + gz^2)^2 \quad\in\sos(3,4)\setminus\{0\}\]
be one square. \ Then by (\ref{eq:taylorSemigroup}) we have
\begin{align*}
&\fp_{f_4}(x,y,z,t)\\
&= f_4(x,y,z) + t\cdot\Delta f_4(x,y,z) + \frac{t^2}{2}\cdot\Delta^2 f_4(x,y,z)\\
&= (ax^2 + bxy + cxz + dy^2 + eyz + gz^2)^2\\
&\quad + 2t\cdot\Big[ \underbrace{(2ax+by+cz)^2 + (bx+2dy+ez)^2 + (cx+ey+2gz)^2}_{=: A_2(x,y,z) = A_2 \in\sos(3,2)} +\\
&\qquad\qquad + 2\cdot(a+d+g)\cdot (ax^2 + bxy + cxz + dy^2 + eyz + gz^2) \Big]\\
&\quad + (\underbrace{12 a^2 + 4 b^2 + 4 c^2 + 8 a d + 12 d^2 + 4 e^2 + 8 a g + 8 d g + 12 g^2}_{=: A_0 > 0\ \text{since}\ f_4\in\pos(3,4)\setminus\{0\}})\cdot t^2.
\end{align*}
For $\lin\{x^2,xy,xz,y^2,yz,z^2\}$ we take a basis $\varepsilon_1,\dots,\varepsilon_6$ with $\varepsilon_1 := ax^2 + bxy + cxz + dy^2 + eyz + gz^2$. \ In the basis $1,x,y,z,\varepsilon_1,\dots,\varepsilon_6$ we can write $\fp_{f_4}(\,\cdot\,,t)$ in a Gram matrix in the form
\begin{equation}\label{eq:f4gram}
\small
\begin{blockarray}{cc|ccc|cccc}
  & 1 & x & y & z & \varepsilon_1 & \varepsilon_2 & \dots & \varepsilon_6\\
  \begin{block}{c(c|ccc|cccc)}
    1 & A_0\cdot t^2 &  &  &  & 2t\cdot (a+d+g) &  &  & \\ \BAhline
x   &  & &  & \\
y   &  & & 2A_2\cdot t & \\
z   &  & & & \\ \BAhline
\varepsilon_1 & 2t\cdot (a+d+g) & & & & 1&\\
\varepsilon_2 &  & & & & & & &\\
\vdots & & & & &  \\
\varepsilon_6 & & & & & & \\
  \end{block}
\end{blockarray}.
\end{equation}
This Gram matrix has block structure. \ For the sub-block of the columns/rows $x,y,z$ we have that $A_2\in\sos(3,2)$ and hence it is positive semi-definite. \ In fact, we have
\begin{equation}\label{eq:a2gram}
A_2(x,y,z)\! =\! \begin{pmatrix} x\\ y\\ z\end{pmatrix}^{\!\!\top}\!\!\cdot\! \begin{pmatrix}
4a^2 + b^2 + c^2 & 2ab+2bd+ce & 2ac+be+2cg\\
2ab+2bd+ce & b^2 + 4d^2 + e^2 & bc+2de+2eg\\
2ac+be+2cg & bc+2de+2eg & c^2 + e^2 + 4g^2
\end{pmatrix} \!\cdot\! \begin{pmatrix} x\\ y\\ z\end{pmatrix}.
\end{equation}
The remaining block from the columns/rows $1$ and $\varepsilon_1$ is
\[A_{0,2}:=\begin{pmatrix}A_0\cdot t^2 & 2t\cdot (a+d+g)\\ 2t\cdot (a+d+g) & 1\end{pmatrix}\]
and we have
\[\det A_{0,2} = t^2\cdot [A_0-4\cdot(a+d+g)^2] = 4t^2\cdot (2a^2+b^2+c^2+2d^2+e^2+2g^2)\]
which shows that $A_{0,2}$ is positive semi-definite and hence our choice (\ref{eq:f4gram}) of the Gram matrix is positive semi-definite and hence a sum of squares representation.

In general, $f_4$ is not a single square, but a sum of ($k\leq 6$) squares:
\[f_4(x,y,z) = \sum_{j=1}^k \alpha_j\cdot (a_j x^2 + b_j xy + c_j xz + d_j y^2 + e_j yz + g_j z^2)^2,\quad \alpha_j>0.\]
By an orthonormal transformation of the Gram matrix of $f_4$ we can assume without loss of generality that $\varepsilon_1 =a_1 x^2 + \dots + g_1 z^2, \dots, \varepsilon_k=a_k x^2 +\dots + g_k z^2$ are orthonormal in $\rset^6$ and completed by $\varepsilon_{k+1},\dots,\varepsilon_6$ to an orthonormal basis of $\rset^6$. \ In this basis one Gram matrix of $\fp_{f}(\,\cdot\,,t)$ has the form
\begin{equation}\label{eq:fGram}
\small
\begin{blockarray}{cc|ccc|cccccc}
  & 1 & x & y & z & \varepsilon_1 & \dots & \varepsilon_k & \varepsilon_{k+1} & \dots & \varepsilon_6 \\
  \begin{block}{c(c|ccc|cccccc)}
    1 & a_{1,1} & a_{x,1} & a_{y,1} & a_{z,1} & * & \dots & * \\ \BAhline
x   & a_{1,x} & a_{x,x} & a_{y,x} & a_{z,x} & * & \dots & * & \\
y   & a_{1,y} & a_{x,y} & a_{y,y} & a_{z,y} & * & \dots & * & \\
z   & a_{1,z} & a_{x,z} & a_{y,z} & a_{z,z} & * & \dots & * & \\ \BAhline
\varepsilon_1 & * & * & * & * & \alpha_1 & & 0 & & & \\
\vdots  & \vdots & \vdots & \vdots & \vdots & & \ddots & & & & \\
\varepsilon_k  & * & * & * & * & 0 & & \alpha_k & & & \\
\varepsilon_{k+1} & & & & & & & & 0 & & \\
\vdots  & & & & & & & & & \ddots & \\
\varepsilon_6 & & & & & & & & & & 0 \\
  \end{block}
\end{blockarray}.
\end{equation}
That the entries of $(1,\varepsilon_1)$ to $(1,\varepsilon_6)$ resp.\ $(\varepsilon_1,1)$ to $(\varepsilon_6,1)$ are zero is because all contributions can be written into the submatrix of the columns/rows $x,y,z$.
That the entries in $(u,v)$ and $(v,u)$ with $u\in\{x,y,z\}$ and $v\in\{\varepsilon_{k+1},\dots,\varepsilon_6\}$ are zero follows from the fact that $f$ is non-negative and the $\varepsilon_1,\dots,\varepsilon_6$ are orthonormal.
Assume an entry $(u,v)$ is non-zero, i.e., it is a homogeneous polynomial of degree $3$. \ 
Since the $\varepsilon_1,\dots,\varepsilon_6$ are orthonormal and the $x^2,xy,\dots,z^2$ are linearly independent, there exists a point $(x_*,y_*,z_*)\in\rset^3$ such that $v(x_*,y_*,z_*)=1$ and $\varepsilon_j(x_*,y_*,z_*)=0$ for all $\varepsilon_j\neq v$.
But then $f(\lambda x_*,\lambda y_*,\lambda z_*) \to -\infty$ for $\lambda\to -\infty$. \ That is a contradiction to $f\in\pos(3,4)$.

For $f_4$ with $k$ squares we can, by an orthonormal coordinate change $(x,y,z)\mapsto (\tilde{x},\tilde{y},\tilde{z})$, always have one square of the form $(a_i x^2 + d_i y^2 + g_i z^2)^2$ with $a_i\neq 0$, and without loss of generality we can assume that $i=1$.
If $d_1 = g_1 = 0$ we can use another coordinate change such that for $i=2$ (if present) we have $(a_2 x^2 + b_2 xy + c_2 xz + d_2 y^2 + g_2 z^2)^2$, i.e., the $y$ and $z$ coordinates are separated.
We now distinguish among four cases:

\underline{(i) $f_4$ depends on $x$, $y$, and $z$:} 
After the previous coordinate change we see that the Gram matrix (\ref{eq:a2gram}) of $A_2$ has full rank and hence for $t\gg 0$ we have that the specific choice (\ref{eq:f4gram}) of the Gram matrix of $\fp_f(\,\cdot\,,t)$ is positive semi-definite and hence $\fp_f(\,\cdot\,,t^*)\in\sos(3,4)$ for some $t^*\geq 0$.

\underline{(ii) $f_4$ depends on only two variables:}
After the previous coordinate changes without loss of generality $f_4$ is independent of $z$.
But since $f\geq 0$, then also $f_3$ (homogeneous part of degree $3$ of $f$) is independent on $z$.
To see this assume, to the contrary that $f_3$ depends on $z^2$.
Since $f_3$ contains only degree $3$ monomials we have that either $xz^2$ or $yz^2$ is in $f_3$.
But in both cases we can chose $(x,y)\in\rset^2$ such that the coefficient of $z^2$ is negative and hence letting $z\to\pm\infty$ gives $f \to -\infty$: This is a contradiction to $f\geq 0$, since $f_4$ has no $z$-dependency for compensation.
So $f_3$ contains no $z^2$.
But the same holds for $z$.
We find $(x,y)\in\rset^2$ such that the coefficient of $z$ is non-zero and either letting $z\to+\infty$ or $z\to-\infty$ gives again $f\to-\infty$, a contradiction to $f\geq 0$.
The linear contributions in $a_{x,x,}$ and $a_{y,y}$ remain (since $f_4$ depends on $x$ and $y$) and hence we again can find, as in (i), a $t^*\geq 0$ such that the Gram matrix representation (\ref{eq:f4gram}) of $\fp_f(\,\cdot\,,t^*)$ is positive definite and hence $\fp_f(\,\cdot\,,t^*)\in\sos(3,4)$.

\underline{(iii) $f_4$ only depends on one variable:}
Without loss of generality $f_4$ depends only on $x$.
Then with the same argument as in (ii) since $f\geq 0$ we have that the specific Gram matrix representation (\ref{eq:f4gram}) of $\fp_f(\,\cdot\,,t^*)$ is positive semi-definite for some $t^*\geq 0$ and hence $\fp_f(\,\cdot\,,t^*)\in\sos(3,4)$.

\underline{(iv) $f_4=0$:} Then also $f_3 =0$ and hence $f\in\pos(3,2) = \sos(3,2)$.
\end{proof}

The \ref{thm:pos34heat} implies the existence of a time when all $f\in\pos(3,4)$ become sum of squares.

\begin{cor}\label{cor:tau34}
There exists a $\tau_{3,4}\in [0,\infty)$ such that
\[\fp_f(\,\cdot\,,t)\in\sos(3,4)\]
for all $t\geq \tau_{3,4}$ and $f\in\pos(3,4)$, i.e., $e^{\tau_{3,4}\Delta} \pos(3,4)\subseteq\sos(3,4)$.
\end{cor}
\begin{proof}
$\fp_f(\,\cdot\,,t)$ is continuous in $f$ and $t$ and since $\pos(3,4)$ is a finite-dimensional cone it has a compact basis. \ Hence,
\[T(f) := \min\{t \,|\, \fp_f(\,\cdot\,,t)\in\sos(3,4)\} <\infty\]
is continuous in $f\in\pos(3,4)$ and therefore
\[\tau_{3,4} := \max_{f\in\pos(3,4)} T(f) < \infty.\qedhere\]
\end{proof}

\begin{cor}
There are $k\in\nset$, $c_1,\dots,c_k\geq 0$, and $y_1,\dots,y_k\in\rset$ such that
\[c_1 p(x + y_1) + \dots + c_k p(x + y_k)\in\sos(3,4)\]
for all $p\in\pos(3,4)$.
\end{cor}
\begin{proof}
The operator $e^{\tau_{3,4}\Delta}$ is a positivity preserver with constant coefficients.
Hence, by \cite[Thm.\ 3.1]{borcea11} there exists a non-negative Borel measure $\mu$ with finite moments on $\rset^2$ such that $e^{\tau_{3,4}\Delta}(p)(x) = \int_{\rset^2} p(y + x)~\diff\mu(y)$.
Since we have the degree bound $\deg p\leq 4$ this integral is a truncated moment functional.
By Richter's Theorem \cite{richte57} we can replace $\mu$ by $\nu = \sum_{i=1}^k c_i \delta_{y_i}$ which proves the statement.
\end{proof}

\begin{cor}
Let $f\in\pos(3,4)$. \ Then there exists a $t=t(f)\in [0,\tau_{3,4}]$ and a $g\in\sos(3,4)$ such that
\[f = \fp_g(\,\cdot\,,-t).\]
\end{cor}

The previous result means that we can generate any $f\in\pos(3,4)$ from a sum of squares by going backwards in the heat equation (\ref{eq:heat}), i.e., taking a negative time in (\ref{eq:taylorSemigroup}).

\begin{cor}
Let $\tau_{3,4}$ be (minimal) as in \Cref{cor:tau34} and $L:\rset[x,y,z]_{\leq 4}\to\rset$ be a linear functional. \ If $\tilde{L} := L\circ e^{-\tau_{3,4}\Delta}$ is strictly square-positive, i.e., $\tilde{L}(g)>0$ for all $g\in\sos(3,4)\setminus\{0\}$, then $L$ is a moment functional.
\end{cor}
\begin{proof}
Let $f\in\pos(3,4)\setminus\{0\}$. \ Then by \Cref{cor:tau34} we have $g = \fp_f(\,\cdot\,,\tau_{3,4}) = e^{\tau_{3,4}\Delta} f \in \sos(3,4)\setminus\{0\}$ and hence
$L(f) = L(e^{-\tau_{3,4}\Delta} e^{\tau_{3,4}\Delta} f) = \tilde{L}(g) > 0$,
i.e., $L$ is in the interior of the truncated moment cone and hence a moment functional.
\end{proof}

Note that the reverse implication in the previous results is in general not true.

\begin{exm}
Let $\tilde{\mu} = \chi_{B_r(0)}\cdot\lambda$ with $r = \sqrt{6\tau_{3,4}}$ and $\lambda$ the Lebesgue measure on $\rset^3$. \ Then $L:\rset[x,y,z]\to\rset$ with representing measure $\mu = e^{\tau_{3,4}\Delta}\chi_{B_r(0)}\cdot\lambda$ is a moment functional with $L(f) > 0$ for $f=x^2 + y^2 + z^2$. \ But $\tilde{f} = e^{-\tau_{3,4}\Delta}f = -6\tau_{3,4}+f$ and hence $\tilde{L}(\tilde{f}) < 0$ since $\tilde{f}\leq 0$ on $B_r(0)$.
\exmsymbol
\end{exm}

\section{Time-dependent set of atoms}
\label{sec:atoms}

\subsection{Coefficients in $C_b^\infty$}

Let $L:\rset[x_1,\dots,x_n]_{\leq d}\to\rset$ with $d\in\nset_0\cup\{\infty\}$ be a (truncated) moment functional. \ When a $p_0\in\rset[x_1,\dots,x_n]_{\leq d}$ with $p_0\geq 0$ exists with $L(p_0) = 0$, then $\supp\mu \subseteq\cZ(p_0)$ for any representing measure $\mu$ of $L$. \ 

In \Cref{sec:dual} we calculated the dual action of (\ref{eq:pde}) on $p_0\in\pol(\rset^n)$. \  It is easy to see that with $\nu>0$ and $p_0\geq 0$, then $p_t>0$ for all $t>0$. \ Therefore, for a moment functional $L$ we always have $L(p_t)>0$ and no restriction of the support of the representing measure is possible. \ Therefore, we can only study the time-dependent set of atoms for $\nu=0$.

Let us remind the reader, that for any time-dependent vector field $g$ and starting point $x_0\in\rset^n$ the system of ordinary differential equations
\begin{equation}\label{eq:forcefield}
\begin{split}
\frac{\diff}{\diff t} G(x,t) &= g(x,t)\\
G(x_0,0) &= x_0
\end{split}
\end{equation}
has, by the Picard--Lindelöf Theorem, a unique solution $G:\rset^n\times\rset\to\rset^n$. \ That means a particle located at $x_0\in\rset^n$ for time $t_0=0$ has the trajectory $G(x_0,t)$ when it experiences the force field $g$, i.e., at time $t\in\rset$ is has the position $G(x_0,t)\in\rset^n$.

The initial value problem (\ref{eq:pde}) acts only on functions $f_0$. \ But by duality in \Cref{sec:dual} also the action of (\ref{eq:pde}) on a measure $\mu_0$ is well-defined. \ The following solves the problem for a Dirac measure.

\begin{lem}\label{lem:atomMovement}
Let $g\in C(\rset,C_b^\infty(\rset^n))^n$ be a time-dependent vector field and $G:\rset^n\times [0,\infty)\to\rset^n$ be the unique solution of (\ref{eq:forcefield}).
Additionally, let $h\in C(\rset,C_b^\infty(\rset^n))$, $x_0\in\rset^n$, and $c_0>0$.
Then the measure-valued differential equation
\begin{equation}\label{eq:measureODE}
\begin{split}
\partial_t \mu_t(x) &= -g(x,t)\nabla \mu_t(x) + h(x,t)\cdot\mu_t(x)\\
\mu_0(x) &= c_0\cdot\delta_{x_0}(x)
\end{split}
\end{equation}
with the initial value $\mu_0 = c_0\delta_{x_0}$ has the unique solution
\[\mu_t(x) = c_0\cdot \exp\left( \int_0^t (h+\divv g)(G(x_0,s),s)~\diff s\right)\cdot \delta_{G(x_0,t)}(x)\]
for all $t\in\rset$.
\end{lem}
\begin{proof}
The operator $-g\nabla + h$ has the dual $g\nabla + h+\divv g$ and we use \Cref{thm:dualP} with $g$ replaced by $-g$.
Let $T>0$ and let $B\subseteq\rset^n$ be a closed ball around $x_0$ such that $G([-T,T],x_0)\subset B$ with $G$ from (\ref{eq:forcefield}).
Take any $p_0\in C^\infty(B,\rset)$, then $\partial_t p = -g\nabla p + (h+\divv g) p$ has a unique classical solution for all $t\in [-T,T]$. \ This solution can be written as
\[p(x,t) = \lim_{N\to\infty} \left[ \prod_{i=N}^0 \exp\left( \int_{t_{i-1}}^{t_i} B(x,s)~\diff s\right) \exp\left( \int_{t_{i-1}}^{t_i} A(x,s)~\diff s\right)\right] p_0(x)\]
for any decomposition $\cZ_N = \{t_0=0 < t_1 < \dots < t_N=t\}$ with $\Delta\cZ_N\to 0$ as $N\to\infty$, with $A = g\nabla$ and $B = h + \divv g$. \ With the approximations
\begin{align*}
\exp\left( \int_{t_0}^{t_1} A(x,s)~\diff s\right) p_0(x) &\approx p_0\left( x + \int_{t_0}^{t_1} g(x,s)~\diff s\right),
\end{align*}
and since $\exp(\int B~\diff s)$ acts by multiplication, we immediately get from
\[\langle p_0,\mu_t\rangle = \langle p(\,\cdot\,,t),\mu_0\rangle\]
the needed statement.
\end{proof}

Note that since $\divv g\neq 0$ in general, the transport equation does not preserve the $L^2$-norm of the solution.
This results in additional scaling with $\divv g$ besides the contribution provided by $h$.

The solution of (\ref{eq:measureODE}) for $\diff\mu_0(x)=f_0(x)~\diff x$ with $f_0\in\cS(\rset^n)$ can in general not be written down explicitly. \ But for the case of $\mu_0 = c_0\cdot\delta_{x_0}$ \Cref{lem:atomMovement} shows that (\ref{eq:measureODE}) is simply solved since the transport term $g(x,t)\cdot\nabla$ acts only on the point $x_0=x(0)$ to get $x(t)$ and the multiplication $h(x,t)$ only acts on the coefficient $c_0=c(0)$ to get $c(t)$. \ This simplifies the description of the solution immensely and provides in the following result a way to trace the Carath\'eodory number $\cat(s)$.

\begin{thm}\label{thm:atomMovements}
Let $n\in\nset$, $s(0)= (s_{\alpha}(0))_{\alpha\in\nset^n}\in \cS_{n,\infty}$ be a $n$-dimensional moment sequence with finite rank Hankel matrix $\cH(s(0))$, $g\in C(\rset,C(\rset^n,\rset^n))$, and $h\in C(\rset,C(\rset,\rset))$. \ Then there is a unique time evolution $s:\rset\to\cS_{n,\infty}$ of $s(0)$ with respect to
\[\partial_t f(x,t) = -g(x,t) \cdot \nabla f(x,t) + h(x,t)\cdot f(x,t).\]
Additionally, for all $t\in\rset$ we have
\[\rank \cH(s(t)) = \rank \cH(s(0))\]
and therefore the Carath\'eodory number is constant, i.e., $\cat(s(t)) = \cat(s(0))$.
\end{thm}
\begin{proof}
Since $K := \rank \cH(s(0))$ is finite, there exists a unique $K$-atomic representing measure $\mu_0 = \sum_{i=0}^K c(0)\cdot \delta_{x_i(0)}$ of $s(0)$. \ By linearity of the time-evolution and \Cref{lem:atomMovement} we have the unique time-evolution of $\mu_t = \sum_{i=1}^K c_i(t)\cdot \delta_{x_i(t)}$ and therefore the unique time-evolution of $s(t)$.

It remains to show that $\cat(s(t)) = \rank \cH(s(t)) = \rank \cH(s(0)) = \cat(s(0))$. \ While the first and the third equality is clear from the rank, see e.g.\ \cite[Prop.\ 17.21]{schmudMomentBook}, we have to show the second equality.
For that it is sufficient to show that the path of $x_k(t)$ never splits or two paths $x_j(t)$ and $x_k(t)$ ($j\neq k$) intersect for any $t\in\rset$.
But this follows from \Cref{lem:atomMovement} and the uniqueness of the solution $x_j(t) = G(x_j(0),t)$ of (\ref{eq:forcefield}) by Peano's Theorem, i.e., when two integral curves $x_j(t)$ and $x_k(t)$ of (\ref{eq:forcefield}) coincide for some $t'\in\rset$, i.e., $x_j(t') = x_k(t')$, then $x_j(t) = x_k(t)$ for all $t\in\rset$, especially for $t=0$ which contradicts the minimal choice of $K = \rank \cH(s(0))$.
\end{proof}

In the univariate case the previous result also holds for the truncated moment problem.

\begin{exm}
Let $s(0) = (s_k(0))_{k=0}^{2d}$ with $d\in\nset$ be a moment sequence on the boundary of the moment cone, i.e., $s$ is represented by $\mu_0 = \sum_{k=1}^{l\leq d} c_k(0)\cdot\delta_{x_k(0)}$ with $c_1(0),\dots,c_l(0)>0$ and $x_1(0),\dots,x_l(0)$ pairwise different and the non-negative polynomial $p_0(x) = (x-x_1(0))^2\cdots (x-x_l(0))^2$ fulfills $L_{s(0)}(p_0) = 0$. \ Let $s(t)$ evolve with respect to $\partial_t f = -g\partial_x f + h\cdot f$. \ By \Cref{lem:atomMovement} we have $\mu_t = \sum_{k=1}^{l\leq d} c_k(t)\cdot\delta_{x_k(t)}$ and therefore $L_{s(t)}(p_t)=0$ for the non-negative polynomial $p_t(x) = (x-x_1(t))^2\cdots (x-x_l(t))^2$. \ Hence, the boundary moment sequence remains a boundary sequence for all times.\exmsymbol
\end{exm}

\begin{rem}
In the previous example we have seen that for univariate truncated moment moment sequences on the boundary the time evolution with respect to (\ref{eq:pde}) with $\nu = 0$ remains a boundary truncated moment sequence for all times. \ And additionally, the time evolution is unique, depending only on the initial moments. \ For multivariate moment sequences this no longer holds. \ Firstly, multivariate (truncated) moment sequences on the boundary can be indeterminate \cite{didioCone22} and the time evolution then depends on the choice of the representing measure $\mu_0$. \ And secondly, even if the truncated boundary moment sequence is determinate, it can immediately enter the interior of the moment cone. \ To see this, take an example from optimal design \cite{reznick92} where the isolated zeros of a non-negative polynomial is maximal. \ For example take $10$ projective zeros of the Robinson polynomial \cite{robinson69} after rotation of the projective space such that no zero lies at infinity (all $10$ zeros are then in the affine space $\rset^2$). \ Since the zeros are isolated we find  a vector field $g:\rset^2\to\rset^2$ such that these points are in general position for $t\neq 0$. \ Then by the Alexander--Hirschowitz theorem \cite{alexa95} there is no non-negative polynomial of degree $6$ vanishing on all these $10$ points, i.e., the truncated moment sequence from the measure $\mu_t$ with $t\neq 0$ belongs to an interior moment sequence.\exmsymbol
\end{rem}

\subsection{Extension to $C^\infty$}
\label{extension}

We have so far treated $g\in C_b^\infty(\rset^n,\rset^n)$. \ We want to see what happens when we extend this class. \ In \cite{curtoHeat22} we calculated the explicit time-evolution of the moments for  $\partial_t f(x,t) = x\cdot \partial_x f(x,t)$ and we got
\[s_k(t) = s_k(0)\cdot e^{-(k+1)\cdot t}.\]
Also for $k=0$ the moments can be calculated only from the initial values $s(0)$ by solving $\partial_t s_0(t) = 0$ and $\partial_t s_k(t) = -k\cdot s_{k-1}(t)$ for $k\in\nset$ by induction.

When we have $\partial_t f = x^k \partial_x f$ with $k\geq 2$, it is in general not possible to calculate the time-dependent moments. \ For moment sequences with finite rank we can at least ensure the existence of the time-evolution $s(t)$ for small times and we observe a finite break down in time.

\begin{exm}\label{exm:finiteBreakDownLinear}
Let $k=2$ and $x_0=1$. \ Then to solve
\begin{align*}
\partial_t \mu_t &= - x^2\cdot\partial_x\mu_t\\ \mu_0 &= \delta_{1}
\end{align*}
we have by \Cref{lem:atomMovement} to solve
\begin{align*}
\partial_t x(t) &= a\cdot x(t)^2\\ x(0) &= 1,
\end{align*}
i.e., we have
\[t = \int_1^x y^{-k}~\diff y = (k-1)\cdot\left[ 1 - x^{1-k}\right]\]
and hence $x(t) = (1-t)^{-1}$. \ Since $\lim_{t\nearrow 1} x(t) = \infty$, we have that the time-dependent moments $s_l(t)$ for $\partial_t f = -x^2\cdot\partial_x f$ exist only for $t<1$ and $s_l(1) = \infty$ for all $l\in\nset$.\exmsymbol
\end{exm}

The previous example demonstrates why we have to restrict to $g\in C_b^\infty$.

\section{Summary and open questions}
\label{sec:summary}

We investigated time-dependent moments from partial differential equations. \ At first we gave the dual action on the polynomials (\Cref{thm:dualP}). \ While for general coefficients $g$ and $h$ we leave $\rset[x_1,\dots,x_n]$ in general, for coefficients independent on $x$ we remain in $\rset[x_1,\dots,x_n]$. \ We investigated the action of the heat equation on non-negative polynomials more deeply. \ In Examples \ref{exm:motzkinCont}, \ref{exm:robinson}, and \ref{exm:schmudgen} to \ref{exm:harris} we have collected examples of non-negative polynomials on $\rset^2$ which are not sums of squares. \ But under the action of the heat equation after finite time they become sums of squares. \ The highest degree part is a non-negative homogeneous polynomial in two variables and hence a sum of squares \cite{hilbert88}. \ If this holds for all bivariate non-negative polynomials is still open.

\begin{open}\label{open1}
Let $f\in\pos(\rset^2)\setminus\sos(\rset^2)$. \ Is it true that there is always a $T = T(f)>0$ such that $\fp_f(x,t)\in\sos(\rset^2)$ for all $t\geq T$?
\end{open}

By \Cref{thm:sos} \Cref{open1} reduces to the questions that $\fp_f(x,T)\in\sos(\rset^2)$ holds for at least one $T>0$.

On $\rset^n$ with $n\geq 3$ this in general no longer holds and the heat equation fails to produce sum of squares from non-negative polynomials in three or more variables, see \Cref{lem:notSOS} and \Cref{exm:nonSOS}.
For the Choi--Lam polynomial we calculated in \Cref{exm:choilam} the exact time when it becomes a sum of squares under the heat equation.
From the counterexamples in \Cref{lem:notSOS} and \Cref{exm:nonSOS} we see that it is necessary that the highest degree part is a sum of squares.
The idea in \Cref{exm:choilam} for the Choi--Lam polynomial is then used in \Cref{thm:pos34heat} to show that any $f\in\pos(3,4)$ becomes a sum of squares in finite time.

The third possibility in Hilbert's Theorem \cite{hilbert88} is $\deg f=2$. \ But then $\fp_f(x,t) = f(x) + c\cdot t$ with $c>0$ \cite{iltyak98}, i.e., homogenization shows already $f\in\sos(n,2)$.

In \Cref{lem:sosAsym} we show that under the heat equation any non-negative polynomial in $\rset[x_1,\dots,x_n]$ gets close to the constant polynomial
\[\lim_{t\to\infty} \fp_f(x,t)\cdot t^{-d} = c > 0,\]
i.e., the positive constant polynomials are attractors of the polynomial heat equation.

For the equation $\partial_t f = g\cdot\nabla f + h\cdot f$ we calculate the time-dependent set of atoms, i.e., we solve the problem of how the atoms of a finitely atomic representing measure $\mu_t = \sum_{i=1}^k c_i(t)\cdot\delta_{x_i(t)}$ evolve in time. \ We find that the atom positions $x_i(t)$ are governed by the transport term $g\cdot\nabla$ but are unaffected by the scaling with $h$. \ The coefficients (masses) $c_i(t)$ can then be analytically solved from the $x_i(t)$, see \Cref{lem:atomMovement} and \Cref{thm:atomMovements}.

We also extended the treatment beyond smooth and bounded coefficients and found that atoms under the transport term $-x^2\cdot\partial_x$ move to infinity in finite time, i.e., a finite break down of the moments (and the measure solution) appears even for linear partial differential equations, see \Cref{exm:finiteBreakDownLinear}.

\section*{Acknowledgments}

The first named author was partially supported by U.S. National Science Foundation research grant DMS-2247167. \ The second named author and this project is financed by the Deutsche Forschungs\-gemeinschaft DFG with the grant DI-2780/2-1 and his research fellowship at the Zukunfs\-kolleg of the University of Konstanz, funded as part of the Excellence Strategy of the German Federal and State Government. \ This work has also been supported by the European Union’s Horizon 2020 research and innovation programme under the Marie Skłodowska-Curie Actions, grant agreement 813211 (POEMA), by the AI Interdisciplinary Institute ANITI funding, through the French ``Investing for the Future PIA3'' program under the Grant agreement n$^\circ$ ANR-19-PI3A-0004 as well as by the National Research Foundation, Prime Minister’s Office, Singapore under its Campus for Research Excellence and Technological Enterprise (CREATE) programme.

We thank the organizers of ICCOPT in Bethlehem, PA, in July 2022 the organizers of IWOTA in Krakow in September 2022, and the Oberwolfach workshop ``Real Algebraic Geometry with a View toward Koopman Operator Methods'' in March 2023 where \cite{curtoHeat22} and the present work have been presented. \ We also thank Tim Netzer for discussions on this work at IWOTA 2022.



\end{document}